\documentclass[onefignum,onetabnum]{siamonline250211}
\usepackage{amsfonts}
\usepackage{amssymb}
\usepackage{mathtools}
\usepackage{graphicx}
\usepackage{epstopdf}
\usepackage{color}
\usepackage{booktabs}
\usepackage{makecell}
\usepackage{algorithm}
\usepackage{algpseudocode}
\usepackage{enumitem}
\ifpdf
  \DeclareGraphicsExtensions{.eps,.pdf,.png,.jpg}
\else
  \DeclareGraphicsExtensions{.eps}
\fi

\providecommand{\citep}[1]{\cite{#1}} \providecommand{\citet}[1]{\cite{#1}}

\makeatletter
\DeclareRobustCommand{\cev}[1]{%
\mathpalette\do@cev{#1}%
}
\newcommand{\do@cev}[2]{%
  \fix@cev{#1}{+}%
\reflectbox{$\m@th#1\vec{\reflectbox{$\fix@cev{#1}{-}\m@th#1#2\fix@cev{#1}{+}$}}$}%
  \fix@cev{#1}{-}%
}
\newcommand{\fix@cev}[2]{%
\ifx#1\displaystyle \mkern#23mu \else \ifx#1\textstyle \mkern#23mu \else \ifx#1\scriptstyle \mkern#22mu \else \mkern#22mu
      \fi
    \fi
  \fi
}
\makeatother

\headers{From Variational Optimization to Flow-Based Transport}{S.~Liang, F.~Bao, H.~G.~Chipilski, P.J.~van Leeuwen, and G.~Zhang}

\title{From Variational Optimization to Flow-Based Transport: Posterior-Geometry Regularization for Ill-Conditioned Data Assimilation
\thanks{This manuscript is authored by UT-Battelle, LLC, under contract DE-AC05-00OR22725 with the US Department of Energy (DOE). The US government retains and the publisher, by accepting the article for publication, acknowledges that the US government retains a nonexclusive, paid-up, irrevocable, worldwide license to publish or reproduce the published form of this manuscript, or allow others to do so, for US government purposes. DOE will provide public access to these results of federally sponsored research in accordance with the DOE Public Access Plan.}
}

\author{Siming Liang\thanks{Computer Science and Mathematics Division, Oak Ridge National Laboratory, Oak Ridge, TN 37831.
}
\and
Feng Bao\thanks{Department of Mathematics, Florida State University, Tallahassee, FL 32306.
}
\and
Hristo G.~Chipilski\thanks{Department of Scientific Computing, Florida State University, Tallahassee, FL 32306. %(\email{hchipilski@fsu.edu})
}
\and
Peter Jan van Leeuwen\thanks{Department of Atmospheric Science, Colorado State University, Fort Collins, CO 80523. %(\email{Peter.vanLeeuwen@colostate.edu})
}
\and
\newline
Guannan Zhang\thanks{Corresponding author. Computer Science and Mathematics Division, Oak Ridge National Laboratory, Oak Ridge, TN 37831 (\email{zhangg@ornl.gov}).}
}

\usepackage{cite}
\definecolor{brown}{rgb}{0.55,0.33,0.14}

\IfFileExists{ulem.sty}{\usepackage[normalem]{ulem}}{}
\newlength{\dellen}
\ifdefined\sout
  
\else
  
\fi

\begin{document}

\maketitle

\begin{abstract}
Variational data assimilation computes a Bayesian update through optimization on a fixed posterior geometry, whose conditioning is determined by the prior covariance and the observation operator. In strongly anisotropic, weakly observed, or highly informative regimes, this geometry can become severely ill conditioned, leading to substantial sensitivity to solver design and preconditioning. This work investigates an alternative computational route based on flow-based transport, using the ensemble score filter as a training-free realization. Rather than repeatedly optimizing over the full posterior geometry, the flow-based update transports the predictive distribution through a sequence of intermediate distributions whose covariance and curvature are progressively regularized before reaching the target posterior. In the linear-Gaussian setting, we characterize this mechanism through the evolution of the diffused posterior covariance and the associated time-dependent curvature, revealing a continuation from a near-isotropic reference geometry to the ill-conditioned posterior geometry targeted by variational assimilation. Numerical stress tests show that this geometric regularization makes the flow-based method substantially less sensitive to severe ill-conditioning than the variational method, while a heterogeneous \(10^4\)-dimensional Lorenz–96 benchmark demonstrates improved robustness and lower delivered computational cost. These results identify posterior-geometry regularization as a key mechanism distinguishing transport-based from optimization-based data assimilation.
\end{abstract}



\section{Introduction}
Data assimilation (DA) combines uncertain model predictions with noisy and incomplete observations to estimate the state of a dynamical system \cite{Carrassi2018,bannister2017}. In real-world applications, the resulting Bayesian posterior can be highly anisotropic because prior uncertainty, observation coverage, and observational accuracy constrain different state directions very unevenly. Such anisotropy is not only a statistical property of the posterior, it also creates a numerical challenge, as widely separated curvature scales can make the Bayesian update step severely ill conditioned.

Variational data assimilation methods perform the update through optimization. In incremental three-dimensional variational data assimilation (3D-Var) and four-dimensional variational data assimilation (4D-Var), nonlinear objectives are locally approximated by quadratic inner-loop problems whose curvature is determined jointly by the background covariance and the linearized observation operator \cite{Barker2004,GrattonLawlessNichols2007,ParrishDerber1992,Rabier2000}. Strong prior anisotropy, weakly observed dimensions, or highly informative observations can therefore lead to ill-conditioned optimization problems. Covariance modeling, including factorizations, balance operators, localization, and hybrid models, determines the background covariance itself \cite{CourtierEtAl1998,DerberBouttier1999,LiGelbLee2024}, while control-variable transforms and Hessian-, Lanczos-, or Krylov-based preconditioners improve the resulting optimization problem \cite{Huang2009,Tremolet2007,FisherEtAl2009}; both are highly effective in practical systems. Ensemble-variational methods construct the background covariance from forecast ensembles and inherit the same incremental structure \cite{lorenc2003potential,hamill2000hybrid,buehner2005ensemble,WangEtAl2013}. Nevertheless, these approaches retain a common computational structure: the Bayesian update is obtained by solving an optimization problem on a fixed local posterior geometry.

 Transport-based filters, which realize the Bayesian update as a coupling of prior and posterior samples \cite{Spantini2022Coupling}, provide a fundamentally different computational framework. Rather than obtaining the analysis by directly minimizing a fixed objective function used in variational DA, score-based filters,  building on score-based generative models whose sampling accuracy in high dimensions has theoretical support \cite{ColeLu2024}, transport samples through pseudo-time between the target posterior distribution and a simple reference distribution \cite{Bao2024EnSF,ZhangBaoZhang2025IEnSF,si2025latent,xiao2026ld}. For example, the ensemble score filter (EnSF) provides a training-free realization of this idea by estimating the score from forecast ensembles and incorporating observations through likelihood-guided reverse transport \cite{BaoEtAl2025,BinderDasguptaOberai2026}; training-free flow-matching filters realize the same idea with deterministic transport \cite{Transue2026FlowMatchingDA}. Although EnSF and related methods have shown promising performance, their relationship to variational DA remains insufficiently understood. In particular, \emph{why can flow-based assimilation remain robust in regimes where variational optimization becomes increasingly sensitive to ill conditioning?}

This work shows that an important distinction lies in the posterior geometry encountered during the update step. Variational DA repeatedly operates on a fixed local posterior geometry, characterized by the posterior-curvature matrix, whereas flow-based DA targets the same Bayesian posterior through a sequence of intermediate distributions whose geometry evolves from an isotropic reference, i.e., the standard Gaussian distribution, toward the ill-conditioned target posterior distribution. We refer to this mechanism as \emph{posterior-geometry regularization}. Here, regularization does not modify the target posterior itself; rather, it regularizes the intermediate geometries encountered on the way to the target posterior. Therefore, flow-based DA methods replace repeated computation on a fixed and ill-conditioned posterior geometry with transport through a progressively evolving geometry.
The distinction between variational and flow-based DA methods can be characterized explicitly in the linear setting. We compare the fixed posterior curvature associated with variational methods with the pseudo-time-dependent curvature generated by flow-based posterior transport, and show how the latter continuously connects an isotropic reference geometry to the target posterior geometry. Our study shows that variational methods can recover the exact Gaussian posterior when sufficiently well solved, while flow-based methods can offer greater robustness to severe conditioning by changing the geometry encountered during the update.

The main contributions of this work are:
% \vspace{0.2cm}
\begin{itemize}[leftmargin=20pt]\itemsep0.1cm
    \item \emph{We establish a posterior-geometry connection between variational optimization and flow-based data assimilation.} In the linear setting, we characterize the fixed posterior geometry characterized by the Hessian of the incremental variational objective and the pseudo-time-dependent curvature induced by flow-based methods, showing how the latter continuously connects an isotropic reference geometry to the target posterior geometry.
\item \emph{We identify posterior-geometry regularization as a mechanism underlying the numerical robustness of flow-based DA methods.} We distinguish the posterior transport from the practical EnSF guidance and analyze how the intermediate curvature and reverse dynamics depend on prior anisotropy and observational information.

\item \emph{We demonstrate the consequences of these different computational geometries in controlled and nonlinear DA problems.} Gaussian stress tests isolate conditioning effects under different prior and observational regimes, while a heterogeneous \(10^4\)-dimensional Lorenz–96 benchmark shows that the flow-based route can provide improved robustness and lower delivered computational cost in a challenging nonlinear setting.
\end{itemize}
\vspace{0.2cm}

The remainder of the paper is organized as follows. Section~\ref{sec:var_da} formulates variational DA in a sequential Bayesian setting and reviews the conditioning of the incremental optimization problem and the role of classical preconditioners. Section~\ref{sec:flow_geometry} derives the exact linear-Gaussian curvature transport and shows how practical EnSF induces a related, computationally accessible geometry path. Section~\ref{sec:numerical_experiments} presents the Gaussian conditioning experiments, the heterogeneous Lorenz–96 benchmark, and computational-cost comparisons. Section~\ref{sec:conclusion} summarizes the implications of posterior-geometry regularization for optimization- and transport-based data assimilation.

\section{Variational Data Assimilation and Fixed Posterior Geometry}
\label{sec:var_da}
We begin by formulating the Bayesian update step that is shared by both variational and flow-based DA methods. The purpose of this section is not to review variational DA in full, but to identify the fixed posterior geometry that governs its incremental optimization process and to clarify how classical preconditioning acts on the geometry.

\subsection{Bayesian update and incremental variational formulation}
\label{sec:var_formulation}

Let $X_n\in\mathbb{R}^d$ denote the hidden model state and $Y_n\in\mathbb{R}^m$ the corresponding observation. We consider the state and observation models
\begin{equation}
    X_{n+1}=\mathcal{M}_n(X_n)+\Xi_n, \qquad
    Y_{n+1}=\mathcal{H}_{n+1}(X_{n+1})+E_{n+1},
    \label{eq:state_obs_model}
\end{equation}
where $\mathcal{M}_n$ is the forecast model, $\mathcal{H}_{n+1}$ is the observation operator, and $\Xi_n$ and $E_{n+1}$ denote model and observation errors. The sequential Bayesian prediction step propagates the current posterior distribution through the model transition density
\begin{equation}
p_{X_{n+1}| Y_{1:n}}(x_{n+1})=\int p_{X_{n+1}| X_n}(x_{n+1}|x_n)p_{X_n| Y_{1:n}}(x_n)\,dx_n,
\label{eq:bayesian_prediction}
\end{equation}
which is the predictive prior at the next assimilation time. After the new observation $y_{n+1}$ becomes available, Bayes' rule gives the posterior filtering density by combining this predictive prior with the observation likelihood, i.e.,
\begin{equation}
p_{X_{n+1}| Y_{1:n+1}}(x_{n+1}) \propto
p_{Y_{n+1}| X_{n+1}}(y_{n+1}| x_{n+1})p_{X_{n+1}| Y_{1:n}}(x_{n+1}).
\label{eq:bayes_update}
\end{equation}
This posterior is the common target posterior distribution considered throughout the paper.

Following standard practice in variational data assimilation, we assume that the predictive prior is a Gaussian distribution given by
\begin{equation}\label{eq:prior}
    X_{n+1}| Y_{1:n} \sim
    \mathcal{N}\!\left(\bar{x}_{n+1| n},\,B_{n+1| n}\right),
\end{equation}
with mean $\bar{x}_{n+1|n}$ and covariance $B_{n+1|n}$, and that the observation errors follow a Gaussian distribution:
\begin{equation}\label{eq:obs_R}
E_{n+1}\sim\mathcal{N}(0,R_{n+1}),
\end{equation}
where $R_{n+1}$ is the observation covariance matrix. The corresponding variational objective function, whose minimizer is the mode (the maximum a posteriori estimate) of the posterior in Eq.~\eqref{eq:bayes_update} \cite{VanLeeuwen1996}, can be defined by
\begin{equation}
\begin{aligned}
    J(x) &= \frac12\left\|x-\bar{x}_{n+1| n}\right\|_{B_{n+1| n}^{-1}}^2 +
    \frac12\left\|y_{n+1}-\mathcal{H}_{n+1}(x)\right\|_{R_{n+1}^{-1}}^2,
\end{aligned}
\label{eq:var_objective}
\end{equation}
where the norm is defined by $\|v\|_M^2:=v^\top Mv$.
Incremental variational DA solves a sequence of local quadratic approximations to Eq.~\eqref{eq:var_objective}. Specifically, we can write
\begin{equation}\label{eq:d}
    x_{n+1}=\bar{x}_{n+1| n}+\delta x, \qquad
    d_{n+1} = y_{n+1}-\mathcal{H}_{n+1}(\bar{x}_{n+1| n}),
\end{equation}
and linearize the observation operator around the reference state, i.e.,
\begin{equation}
    \mathcal{H}_{n+1}(\bar{x}_{n+1| n}+\delta x) \approx
    \mathcal{H}_{n+1}(\bar{x}_{n+1| n}) + H_{n+1}\delta x,
\end{equation}
where $H_{n+1}$ is the Jacobian of the observation operator. In nonlinear applications, the approximation is updated through outer-loop iterations, so the method should be viewed as a sequence of local quadratic models rather than a single globally linear inverse problem, following the standard Gauss--Newton scheme for nonlinear least-squares problems \cite{GrattonLawlessNichols2007}. The corresponding inner-loop problem becomes
\begin{equation}
    \delta x^\star =\arg\min_{\delta x}\left\{ J(\delta x)\right\}
    = \arg\min_{\delta x} \left\{ \frac12\|\delta x\|_{B_{n+1| n}^{-1}}^2 + \frac12\|d_{n+1}-H_{n+1}\delta x\|_{R_{n+1}^{-1}}^2 \right\}.
    \label{eq:incremental_var}
\end{equation}
Its gradient and Gauss--Newton Hessian are
\begin{equation}
\begin{aligned}
    \nabla_{\delta x}J &= B_{n+1| n}^{-1}\delta x - H_{n+1}^{\top}R_{n+1}^{-1} \left(d_{n+1}-H_{n+1}\delta x\right),\\
    \nabla_{\delta x}^2J &=B_{n+1| n}^{-1}+ H_{n+1}^{\top}R_{n+1}^{-1}H_{n+1}.
\end{aligned}
\label{eq:var_grad_hess}
\end{equation}
Following the Gauss--Newton formalism, the term involving second derivatives of $\mathcal{H}_{n+1}$ is neglected in Eq.~\eqref{eq:var_grad_hess}; for a linear observation operator the expression is exact \cite{GrattonLawlessNichols2007}, and the matrix $\nabla_{\delta x}^2J$ is simultaneously the Hessian of the variational objective, the precision matrix of the Gaussian posterior, and the posterior-curvature matrix. We use posterior geometry to refer more broadly to the anisotropic structure characterized locally by this curvature. For a nonlinear observation operator, it is the local Gauss--Newton curvature used in the current inner loop defined by Eq.~\eqref{eq:incremental_var}. In variational DA, the Hessian $\nabla_{\delta x}^2J$ in Eq.~\eqref{eq:var_grad_hess} is the central posterior-curvature descriptor that defines the geometry encountered by variational optimization.

\subsection{Ill-conditioned posterior geometry}
\label{sec:ill_conditioned_geometry}
For notational simplicity, we define the following shorthand for the matrix $\nabla_{\delta x}^2J$ in Eq.~\eqref{eq:var_grad_hess}:
\begin{equation}
\boxed{
A:=B_{n+1| n}^{-1} + H_{n+1}^{\top}R_{n+1}^{-1}H_{n+1}.}
    \label{eq:posterior_curvature_split}
\end{equation}
Throughout the remainder of the paper, we refer to \(A\) as the posterior-curvature matrix. Poor conditioning of the Bayesian update is a property of the \emph{full posterior geometry}, not of the background covariance alone. Strong prior anisotropy can produce widely separated curvature scales, but so can sparse or highly precise observations. To make the numerical consequence explicit, we consider the quadratic form obtained by expanding the inner-loop objective in Eq.~\eqref{eq:incremental_var}:
\begin{equation}
    J(\delta x) =
    \frac{1}{2}\delta x^\top A\delta x - g^\top\delta x + c,
\end{equation}
where $A$ is a symmetric positive definite matrix. Here $A$ is precisely the posterior-curvature matrix in Eq.~\eqref{eq:posterior_curvature_split}; $g=H_{n+1}^{\top}R_{n+1}^{-1}d_{n+1}$ is the innovation mapped into state space; and $c=\tfrac{1}{2}\,d_{n+1}^{\top}R_{n+1}^{-1}d_{n+1}$ is a constant independent of $\delta x$, which shifts the value of $J$ but not its minimizer $\delta x^{\star}=A^{-1}g$. The stability and convergence behavior analyzed below depend only on $A$. Fixed-step gradient descent satisfies $\delta x_{k+1}=\delta x_k-\eta(A\delta x_k-g),$ and is stable only if \cite{NocedalWright2006}

\begin{equation}
    0<\eta<\frac{2}{\lambda_{\max}(A)}.
    \label{eq:gd_stability}
\end{equation}
For the optimal fixed step size
\begin{equation}
    \eta^\star=\frac{2}{\lambda_{\max}(A)+\lambda_{\min}(A)},
\end{equation}
the asymptotic error contraction factor is
\begin{equation}
    \rho^\star = \frac{\kappa(A)-1}{\kappa(A)+1},
    \qquad
    \kappa(A)=\frac{\lambda_{\max}(A)}{\lambda_{\min}(A)}.
    \label{eq:gd_contraction}
\end{equation}
Therefore, as $\kappa(A)$ grows, a fixed first-order iteration becomes increasingly slow and increasingly sensitive to step-size selection. This calculation is not intended to suggest that all variational solvers behave like gradient descent; line searches and quasi-Newton methods \cite{GilbertLemarechal1989}, conjugate-gradient and Lanczos methods \cite{FisherEtAl2009}, and preconditioned Krylov solvers \cite{TshimangaEtAl2008,GrattonTshimanga2009,GurolEtAl2014} can substantially improve convergence. Rather, Eq.~\eqref{eq:gd_stability}--Eq.~\eqref{eq:gd_contraction} make explicit why spectral anisotropy of the posterior curvature is the central numerical quantity that preconditioned variational methods must address.

\subsection{Classical preconditioning of the fixed posterior geometry}
\label{sec:classical_preconditioning}

A standard first-level preconditioning strategy is the control-variable transform (CVT). Specifically, we define
\begin{equation}
    \delta x=Uv, \qquad
    B_{n+1| n}=UU^\top,
    \label{eq:cvt}
\end{equation}
where $U$ is a square root, approximate square root, or structured factor of the background covariance matrix. Substituting Eq.~\eqref{eq:cvt} into Eq.~\eqref{eq:incremental_var}, we convert the quadratic problem in Eq.~\eqref{eq:incremental_var} to the following problem:
\begin{equation}
    v^\star=\arg\min_v\left\{\|v\|_2^2+ \|d_{n+1}-H_{n+1}Uv\|_{R_{n+1}^{-1}}^2 \right\},
    \label{eq:cvt_objective}
\end{equation}
where the Hessian becomes
\begin{equation}
    A_{\mathrm{CVT}}=I+U^\top H_{n+1}^{\top}R_{n+1}^{-1}H_{n+1}U.
    \label{eq:cvt_hessian}
\end{equation}
Eq.~\eqref{eq:cvt_hessian} shows both the strength and the limitation of first-level preconditioning. The CVT removes the background contribution $B_{n+1| n}^{-1}$ from the Hessian and replaces it by the identity, thereby eliminating prior-induced anisotropy in the transformed coordinates. Since the observation-induced term is positive semidefinite, all eigenvalues of Eq.~\eqref{eq:cvt_hessian} are bounded below by one, and when the posterior anisotropy is dominated by the prior, the conditioning of the transformed problem improves dramatically. However, the observation-induced term $U^\top H_{n+1}^{\top}R_{n+1}^{-1}H_{n+1}U$ remains. The transformed problem can therefore still be strongly anisotropic, particularly when precise observations act on high-variance directions of the prior \cite{HabenLawlessNichols2011}. In this sense, prior whitening is not equivalent to posterior whitening.

Second-level preconditioners based on approximate Hessian information, Lanczos vectors, Ritz pairs, or Krylov subspaces can further reduce the effective conditioning of Eq.~\eqref{eq:cvt_hessian}, and can solve linear-Gaussian inner problems very efficiently \cite{TshimangaEtAl2008,GrattonTshimanga2009,FisherEtAl2009}. The distinction pursued in this paper is therefore not that ill-conditioned variational problems are intrinsically unsolvable, nor that classical preconditioning is ineffective. Rather, classical preconditioning changes the coordinates or effective spectrum of a \emph{fixed} local optimization problem. Within each inner loop, the solver still repeatedly acts on a single posterior-curvature operator, whether represented by $A$ in state space or by $A_{\mathrm{CVT}}$ in control space.

This fixed-geometry viewpoint provides the reference point for the flow-based construction developed next. Instead of repeatedly solving on one posterior curvature, flow-based DA reaches the same Bayesian target through a pseudo-time sequence of intermediate distributions with evolving curvature. Section \ref{sec:flow_geometry} shows how this change from fixed optimization to time-dependent transport leads naturally to the notion of posterior-geometry regularization.
\section{Flow-Based DA with Posterior-Geometry Regularization}
\label{sec:flow_geometry}

Section~\ref{sec:var_da} shows that incremental variational data assimilation repeatedly acts on a fixed local posterior-curvature operator, i.e., the Hessian matrix in Eq.~\eqref{eq:posterior_curvature_split}. We now investigate a different computational route to the same Bayesian update. Flow-based DA introduces a pseudo-time family of distributions connecting the target distribution to a simple Gaussian reference distribution and generates posterior samples through reverse transport. We use the ensemble score filter (EnSF) \cite{BaoEtAl2025,Bao2024EnSF,ZhangBaoZhang2025IEnSF} as a training-free realization of this construction. Our focus is the geometry induced by EnSF's pseudo-time transport and its relationship to the fixed posterior geometry used in variational optimization.

\subsection{Flow-based transport and EnSF}
\label{sec:ensf_transport}
Throughout, the state is assumed nondimensionalized componentwise, as is standard in EnSF implementations \cite{BaoEtAl2025,Bao2024EnSF}. This removes physical units but not cross-variable correlation, so the covariance of the scaled state can remain severely ill conditioned.

\subsubsection{The diffusion process for the prior distribution}\label{sec:EnSF_prior}
To compare variational optimization and flow-based transport on the same Bayesian update, we formulate EnSF in the increment space introduced in Section~\ref{sec:var_formulation}. We define the state increment as
\begin{equation}
    \Delta X := X_{n+1|n} -  \bar{x}_{n+1| n},
    \label{eq:forecast_increment}
\end{equation}
where $\bar{x}_{n+1\mid n}$ is the forecast mean given in Eq.~\eqref{eq:prior}. The predictive prior in increment space is denoted by $p_{\Delta X}$ and has zero mean by construction. Following the notations used in the standard EnSF framework\cite{BaoEtAl2025,Bao2024EnSF}, define a forward diffusion process on the pseudo-time interval $t\in[0,1]$ for the prior by
\begin{equation}
{\rm d} Z_{n+1|n,t}=b_t Z_{n+1|n,t}\,{\rm d}t+\sigma_t\,{\rm d}W_t, \qquad Z_{n+1|n,0} \sim p_{\Delta X},
    \label{eq:forward_sde}
\end{equation}
where $W_t$ is a standard Brownian motion and the subscript $(\cdot)_{n+1|n}$ represents that the stochastic process is built for the prior distribution $p_{\Delta X}$. The coefficients are parameterized
through a noise schedule $(\alpha_t,\beta_t)$ as
\begin{equation}
    b_t=\frac{d\log\alpha_t}{dt}, \qquad
    \sigma_t^2 = \frac{d\beta_t^2}{dt} - 2\frac{d\log\alpha_t}{dt}\beta_t^2.
    \label{eq:sde_coeff}
\end{equation}
The marginal distribution of Eq.~\eqref{eq:forward_sde} admits the representation 
\begin{equation}
    Z_{n+1|n,t} \overset{d}{=} \alpha_t Z_{n+1|n,0} + \beta_t\varepsilon, \qquad
    \varepsilon\sim\mathcal{N}(0,I),
    \label{eq:diffusion_marginal}
\end{equation}
with $\varepsilon$ independent of $Z_{n+1|n,0} = \Delta X$, so that $Z_{n+1|n,t}$ has the covariance matrix
\begin{equation}
\Sigma_t=\alpha_t^2B_{n+1| n}+\beta_t^2I,
\label{eq:sec5_diffused_cov}
\end{equation}
where $B_{n+1| n}$ is the background covariance matrix in Eq.~\eqref{eq:prior}. Throughout this paper including the numerical experiments we use the following definitions of $\alpha_t$ and $\beta_t$:
\begin{equation}
    \alpha_t=1-t, \qquad \beta_t^2=t, \qquad 0\le t\le 1,
    \label{eq:noise_schedule}
\end{equation}
so that the forward process moves from the target distribution at $t=0$ toward a standard Gaussian reference distribution at $t=1$. In this case, the prior score associated with $Z_{n+1|n,t}$ in Eq.~\eqref{eq:diffusion_marginal} can be written as
\begin{equation}
S_{n+1|n}(z,t)
=-\Sigma_t^{-1}\left(z - \alpha_t\, \mathbb{E}[\Delta X]\right)
=-\Sigma_t^{-1} z,
\label{eq:sec5_diffused_score}
\end{equation}
where the second equality holds because the increment is centered by construction, i.e., $\mathbb{E}[\Delta X]=0$. If a reference state other than the prior mean is used, the mean term must be retained. In practice the score is evaluated from the forecast ensemble through the sample covariance in Eq.~\eqref{eq:sec5_diffused_score} or the Monte Carlo estimator of \cite{Bao2024EnSF} without training a neural network, which is why this realization is training-free.

\subsubsection{The diffusion process for the posterior distribution}\label{sec:EnSF_post}
The central task in flow-based DA is to draw samples from the posterior distribution of the increment by mapping a standard Gaussian distribution to posterior distribution via a reverse-time SDE:
\begin{equation}
{\rm d}{Z}_{n+1|n+1,t}=
\left(b_t {Z}_{n+1|n+1,t}-\sigma_t^2S_{n+1|n+1}({Z}_{n+1|n+1,t},t)\right){\rm d}t+ \sigma_t\,{\rm d}\cev{W}_t,
\label{eq:ensf_reverse_sde_general_sec4}
\end{equation}
where $\cev{W}_t$ is a Brownian motion running backward in pseudo-time, and the subscript $(\cdot)_{n+1|n+1}$ represents that the stochastic process is built for the posterior distribution. To do this, we need to compute the posterior score function $S_{n+1|n+1}$ in Eq.~\eqref{eq:ensf_reverse_sde_general_sec4}. EnSF updates the prior score to the posterior score by adding likelihood guidance in the diffused space, written as
\begin{equation}
\begin{aligned}
    {S}_{n+1|n+1}(z,t)& ={S}_{n+1|n}(z,t)+h(t)\,\nabla_z\log p_{Y_{n+1}|\Delta X}(y_{n+1}|z)\\[5pt]
    & = -\Sigma_t^{-1} z -h(t) \nabla_z\mathcal{H}(z)^TR_{n+1}^{-1} (\mathcal{H}(z) - d_{n+1}),
\end{aligned}
\label{eq:ensf_posterior_score_baseline_sec4}
\end{equation}
where $h(t)$ is a decreasing damping function satisfying $h(0)=1$, $h(1)=0$, $R_{n+1}$ is the covariance matrix of the observation error, and $\mathcal{H}$ is the observation operator. This is the working EnSF formula used for posterior sampling: the prior score provides the generative structure in the reference space, while the likelihood score guides that structure toward consistency with the new data. At $t=0$, the score decomposition corresponds to the usual Bayesian decomposition of prior score plus likelihood score. Away from $t=0$, \eqref{eq:ensf_posterior_score_baseline_sec4} should be interpreted as a likelihood-guidance mechanism rather than a general exact nonlinear posterior score.

\begin{remark}[nonlinear observation operators and increment-space notation]
    Here $\mathcal{H}$ acts on the state increment through $\mathcal{H}(z):=\mathcal{H}_{n+1}(\bar{x}_{n+1\mid n}+z)-\mathcal{H}_{n+1}(\bar{x}_{n+1\mid n})$, so that $\mathcal{H}(z)-d_{n+1}=\mathcal{H}_{n+1}(\bar{x}_{n+1\mid n}+z)-y_{n+1}$. The likelihood guidance in Eq.~\eqref{eq:ensf_posterior_score_baseline_sec4} accepts a nonlinear observation operator directly, whereas incremental variational DA must rely on the local linearization of Section~\ref{sec:var_formulation}; when the linearization of $\mathcal{H}_{n+1}$ at $\bar{x}_{n+1\mid n}$ is not accurate across the range of states covered by the prior uncertainty, it introduces a second bottleneck beyond ill conditioning, since the Gauss--Newton curvature $A$ may no longer reflect the true posterior geometry. In previous work, the performance of EnSF under nonlinear observation operators has been shown to be close to its performance under linear ones \cite{BaoEtAl2025,liang_ensf_inpainting_2025,xiong2025sensitivity}. Since the focus of this paper is the ill-conditioned background covariance, the analysis below uses a linear observation operator, for which $\mathcal{H}(z)=H_{n+1}z$, and Sections~\ref{sec:exact_posterior_transport} and~\ref{sec:ensf_geometry_transport} should be read in that setting.
\end{remark}

\subsection{Posterior transport in the linear-Gaussian setting}
\label{sec:exact_posterior_transport}

To illustrate the posterior geometry of flow-based DA, we consider the linear-Gaussian
incremental problem from Section~\ref{sec:var_formulation}, meaning the observation operator is a linear operator, i.e.,

\begin{equation}\label{eq:lin_obs}
    \mathcal{H}_{n+1}(\bar{x}_{n+1\mid n}+\Delta X) = \mathcal{H}_{n+1}(\bar{x}_{n+1\mid n}) + H_{n+1}\Delta X,
\end{equation}
where $\Delta X$ is defined in Eq.~\eqref{eq:forecast_increment}, and the increment-space operator of the remark above reduces to $\mathcal{H}(z)=H_{n+1}z$. In this case, the exact posterior of $\Delta X$ is also a Gaussian distribution defined by
\begin{equation}
     \Delta X|d_{n+1} \sim \mathcal{N}(\mu,A^{-1})\;\; \text{ with }\;
    \mu=A^{-1} H_{n+1}^\top R_{n+1}^{-1}d_{n+1},
    \label{eq:exact_gaussian_posterior}
\end{equation}
where $d_{n+1}$ is the innovation defined in Eq.~\eqref{eq:d}, $R_{n+1}$ is the observation-error covariance defined in Eq.~\eqref{eq:obs_R}, and the curvature matrix $A$ is defined in Eq.~\eqref{eq:posterior_curvature_split}.

When the diffusion process of Section~\ref{sec:ensf_transport} is applied to the posterior distribution in Eq.~\eqref{eq:exact_gaussian_posterior}, i.e., with initial state $Z_{n+1\mid n+1,0}\sim\mathcal{N}(\mu,A^{-1})$, the exact intermediate states are
\begin{equation}
    Z_{n+1\mid n+1,t} = \alpha_t Z_{n+1\mid n+1,0}+\beta_t\varepsilon, \qquad
    \varepsilon\sim\mathcal{N}(0,I),
    \label{eq:diffused_posterior}
\end{equation}
with $\varepsilon$ independent of $Z_{n+1\mid n+1,0}$, so that
\begin{equation}
    Z_{n+1\mid n+1,t} \sim \mathcal{N}
    \left(\alpha_t \mu,\, \alpha_t^2 A^{-1} + \beta_t^2 I \right).
    \label{eq:diffused_posterior_cov}
\end{equation}
Hence the exact posterior score function along pseudo-time is
\begin{equation}
    S_{n+1|n+1,t}^\star(z,t) = - \left(\alpha_t^2 A^{-1} + \beta_t^2 I\right)^{-1} \left(z-\alpha_t \mu\right),
    \label{eq:exact_posterior_score}
\end{equation}
and the associated pseudo-time-dependent posterior curvature is
\begin{equation}
    \boxed{ A_t^\star :=  \left(  \alpha_t^2A^{-1}  +  \beta_t^2I \right)^{-1}.}
    \label{eq:exact_posterior_curvature}
\end{equation}

Eq.~\eqref{eq:exact_posterior_curvature} is the central geometric object for the flow-based route. Variational optimization repeatedly encounters the endpoint curvature $A$, whereas exact reverse transport starts from a near-isotropic reference geometry ($\beta_t \rightarrow 1$ and $\alpha_t \rightarrow 0$) and progressively introduces the anisotropy of $A$ ($\beta_t \rightarrow 0$ and $\alpha_t \rightarrow 1$).

\begin{proposition}[Regularization of posterior curvature]
\label{prop:posterior_curvature}
Let $A\succ0$ be the posterior curvature defined in Eq.~\eqref{eq:posterior_curvature_split} and let $A_t^\star$ be the pseudo-time-dependent posterior diffusion curvature defined in Eq.~\eqref{eq:exact_posterior_curvature}. If $A$ has eigendecomposition $A=Q\Lambda Q^\top$ with $\Lambda=\operatorname{diag}(\lambda_1,\ldots,\lambda_d)$ and  $\lambda_{\max}\ge \lambda_i\ge \lambda_{\min}>0$, then $A_t^\star$ has the following properties:
\vspace{0.2cm}
\begin{itemize}[leftmargin=20pt]
    \item The eigenvalues of $A_t^\star$ are
    \begin{equation}
    \lambda_i^\star(t)   = \frac{\lambda_i}{\alpha_t^2 + \beta_t^2\lambda_i}.
    \label{eq:exact_curvature_eigs}
\end{equation}
\item The condition number of $A_t^\star$ is
\begin{equation}
    \kappa(A_t^\star)  =  \kappa(A) \frac{\alpha_t^2+\beta_t^2\lambda_{\min}}{\alpha_t^2+\beta_t^2\lambda_{\max}}.
    \label{eq:exact_curvature_condition}
\end{equation}
\item Consequently, $\kappa(A_t^\star)\le\kappa(A)$ for all $t\in[0,1]$, with $\kappa(A_0^\star)=\kappa(A)$ and $\kappa(A_1^\star)=1$; under the schedule in Eq.~\eqref{eq:noise_schedule}, $\kappa(A_t^\star)$ decreases monotonically in $t$.
\end{itemize}
\end{proposition}

The proof of the proposition is straightforward by substituting the eigendecomposition of $A$ into the posterior curvature in Eq.~\eqref{eq:exact_posterior_curvature}. Since $A^{-1}=Q\,\operatorname{diag}(\lambda_1^{-1},\ldots,\lambda_d^{-1})\,Q^\top$ and $QQ^\top=I$,
\begin{equation}
    A_t^\star
    = Q\left(\alpha_t^2\Lambda^{-1}+\beta_t^2 I\right)^{-1}Q^\top
    = Q\,\operatorname{diag}\!\left(\frac{\lambda_1}{\alpha_t^2+\beta_t^2\lambda_1},\ldots,
      \frac{\lambda_d}{\alpha_t^2+\beta_t^2\lambda_d}\right)Q^\top.
    \label{eq:exact_curvature_substitution}
\end{equation}
 This gives Eq.~\eqref{eq:exact_curvature_eigs} directly and the ratio of $\lambda^\star_{\max}(t)$ and $\lambda^\star_{\min}(t)$ yields Eq.~\eqref{eq:exact_curvature_condition}.
An additional consequence is that $A_t^\star$ is a matrix function of $A$. Therefore all intermediate exact posterior curvatures commute and share the eigenvectors of the target posterior. In the linear setting, the exact flow does not rotate the stiff posterior directions; it progressively restores their relative curvature scales along the reverse stochastic process.

\subsection{Posterior-geometry transport induced by EnSF}
\label{sec:ensf_geometry_transport}

Section~\ref{sec:exact_posterior_transport} shows that, in the linear-Gaussian setting, diffusing the exact posterior distribution generates the pseudo-time-dependent posterior curvature defined in Eq.~\eqref{eq:exact_posterior_curvature}, which continuously connects the target posterior curvature $A$ at $t=0$ to the isotropic Gaussian curvature $I$ at $t=1$. The practical EnSF introduced in Section~\ref{sec:ensf_transport} defines a related but different posterior-geometry transport.

For a linear observation operator, Eq.~\eqref{eq:ensf_posterior_score_baseline_sec4} can be written as
\begin{equation}
\begin{aligned}
    S_{n+1\mid n+1}(z,t) &= -\Sigma_t^{-1}z - h(t) H_{n+1}^{\top} R_{n+1}^{-1}\left(H_{n+1}z-d_{n+1}\right)\\
    &= -\left(\Sigma_t^{-1} + h(t)H_{n+1}^{\top} R_{n+1}^{-1}H_{n+1}\right) z + h(t)\,H_{n+1}^{\top}R_{n+1}^{-1}d_{n+1}\\
    &=-A_t^{\mathrm{EnSF}}z+h(t)H_{n+1}^{\top}R_{n+1}^{-1}d_{n+1},
\end{aligned}
\label{eq:ensf_score_sec33}
\end{equation}
where the curvature matrix induced by EnSF is
\begin{equation}
    \boxed{A_t^{\mathrm{EnSF}}=\left(\alpha_t^2 B_{n+1\mid n}+\beta_t^2 I\right)^{-1}+h(t)H_{n+1}^{\top}R_{n+1}^{-1}H_{n+1}.}
    \label{eq:ensf_curvature_sec33}
\end{equation}
Thus EnSF also defines a pseudo-time-dependent posterior curvature, whose geometry can be characterized as follows.

\begin{proposition}[Regularization of the EnSF-induced curvature]
\label{prop:ensf_curvature}
Let the curvature matrix $A_t^{\mathrm{EnSF}}$ be defined in Eq.~\eqref{eq:ensf_curvature_sec33} with the schedule given in Eq.~\eqref{eq:noise_schedule} and a damping function $h(t)$ satisfying $h(0)=1$ and $h(1)=0$. Let $b_{\max}\ge b_{\min}>0$ denote the extreme eigenvalues of $B_{n+1\mid n}$ and set $G:=H_{n+1}^{\top}R_{n+1}^{-1}H_{n+1}\succeq0$. Then $A_t^{\mathrm{EnSF}}$ has the following properties:
\vspace{0.2cm}
\begin{itemize}[leftmargin=20pt]\itemsep0.2cm
    \item \emph{Endpoint consistency:} $ A_0^{\mathrm{EnSF}}=B_{n+1\mid n}^{-1}+H_{n+1}^{\top}R_{n+1}^{-1}H_{n+1}=A$ and $A_1^{\mathrm{EnSF}}=I$, so the EnSF's curvature transport path connects the same reference and target geometries as the exact path $A_t^\star$ in Eq.~\eqref{eq:exact_posterior_curvature}.
    \item \emph{Two-sided spectral bounds:} $\Sigma_t^{-1}+h(t)\lambda_{\min}(G)\,I \preceq A_t^{\mathrm{EnSF}}\preceq \Sigma_t^{-1}+h(t)\lambda_{\max}(G)\,I$, and hence every eigenvalue of $A_t^{\mathrm{EnSF}}$ satisfies
    \begin{equation}
        \frac{1}{\alpha_t^2 b_{\max}+\beta_t^2}+h(t)\,\lambda_{\min}(G) \;\le\;
        \lambda_i\!\left(A_t^{\mathrm{EnSF}}\right) \;\le\;
        \frac{1}{\alpha_t^2 b_{\min}+\beta_t^2}+h(t)\,\lambda_{\max}(G).
        \label{eq:ensf_curvature_eig_bounds}
    \end{equation}
    \item \emph{Condition-number bound:}
    \begin{equation}
    \begin{aligned}
        &\kappa\!\left(A_t^{\mathrm{EnSF}}\right)\;\le\; \kappa(\Sigma_t)+h(t)\,\lambda_{\max}(G)\,\big(\alpha_t^2 b_{\max}+\beta_t^2\big),\\[4pt]
        &\hspace{1.6cm}\kappa(\Sigma_t)=\frac{\alpha_t^2 b_{\max}+\beta_t^2}{\alpha_t^2 b_{\min}+\beta_t^2}\le\kappa\!\left(B_{n+1\mid n}\right),
    \end{aligned}
        \label{eq:ensf_curvature_cond_bound}
    \end{equation}
    so that $\kappa(A_t^{\mathrm{EnSF}})\to1$ as $t\to1$, meaning that the reverse transport starts on a near-isotropic geometry regardless of how anisotropic the prior is and how informative the observation is.
\end{itemize}
\end{proposition}

\begin{proof}
The endpoint values follow from $(\alpha_0,\beta_0,h(0))=(1,0,1)$ and $(\alpha_1,\beta_1,h(1))=(0,1,0)$. The two-sided bound follows from $h(t)\lambda_{\min}(G)I \preceq h(t)G\preceq h(t)\lambda_{\max}(G)I$ and Weyl's monotonicity theorem applied to $A_t^{\mathrm{EnSF}}=\Sigma_t^{-1}+h(t)G$, using $\lambda_{\max}(\Sigma_t^{-1})=(\alpha_t^2b_{\min}+\beta_t^2)^{-1}$ and $\lambda_{\min}(\Sigma_t^{-1})=(\alpha_t^2b_{\max}+\beta_t^2)^{-1}$. Dividing the upper bound on $\lambda_{\max}(A_t^{\mathrm{EnSF}})$ by the lower bound on $\lambda_{\min}(A_t^{\mathrm{EnSF}})$ and dropping the nonnegative term $h(t)\lambda_{\min}(G)$ yields the condition-number bound.
\end{proof}

Proposition~\ref{prop:ensf_curvature} extends the regularization mechanism of Proposition~\ref{prop:posterior_curvature} to the practical EnSF path. The damping function plays for the observation-induced curvature the same role that the diffusion plays for the prior-induced curvature: it switches the term $h(t)H_{n+1}^{\top}R_{n+1}^{-1}H_{n+1}$ off at the reference end, so the reverse transport begins on a near-isotropic geometry no matter how informative the observation is. The second term of the bound in Eq.~\eqref{eq:ensf_curvature_cond_bound} quantifies the price of carrying observation information along the path: it is largest when precise observations act on high-variance prior directions, in which case $\lambda_{\max}(G)\,b_{\max}$ is of the order of $B_{00}/\sigma_{\rm obs}^2$, i.e., the stiff mode that dominates the fixed geometries $A$ and $A_{\mathrm{CVT}}$ in the experiments of Section~\ref{sec:numerical_experiments}, while the transport tames it through the damping factor $h(t)$. Unlike $A_t^\star$, the matrix $A_t^{\mathrm{EnSF}}$ is generally not a matrix function of $A$: the diffused-prior term and the damped observation term do not commute, so the EnSF path may rotate eigenvectors along the way, and the proposition provides bounds rather than identities; the commuting case shows that the bounds are sharp.

Therefore, viewed in forward pseudo-time, EnSF transports the posterior curvature from $A$ toward the standard-Gaussian curvature $I$; viewed in the reverse direction used for posterior sampling, it follows
\begin{equation}
    I \longrightarrow A_t^{\mathrm{EnSF}} \longrightarrow A \;\; \text{ as }\;\; t \rightarrow 0.
\label{eq:ensf_geometry_path_sec33}
\end{equation}
In this sense, the practical EnSF retains the central geometric mechanism identified in Section~\ref{sec:exact_posterior_transport}: it replaces repeated computation on the fixed posterior curvature by transport through a sequence of intermediate geometries.

In this sense, the practical EnSF retains the central geometric mechanism identified in Section~\ref{sec:exact_posterior_transport}: it replaces repeated computation on the fixed posterior curvature by transport through a sequence of intermediate geometries. The exact path $A_t^\star$ is the theoretical reference for this mechanism, while $A_t^{\mathrm{EnSF}}$ is its practical realization: it is built from the diffused prior and the damped likelihood alone, so its intermediate geometry can be evaluated without the posterior covariance $A^{-1}$.

The EnSF path in Eq.~\eqref{eq:ensf_curvature_sec33} is generally not identical to the exact curvature path in Eq.~\eqref{eq:exact_posterior_curvature} for $0<t<1$. The exact path depends on the posterior covariance $A^{-1}$, whereas the EnSF path is constructed directly from the diffused-prior geometry $\left(\alpha_t^2 B_{n+1\mid n}+\beta_t^2 I\right)^{-1}$ and the damped likelihood information. By Proposition~\ref{prop:ensf_curvature}, the two paths nevertheless share the same endpoint curvature matrices. The exact path therefore provides a theoretical reference for posterior-geometry regularization, while EnSF provides a practical realization whose intermediate geometry can be evaluated without explicitly constructing the exact diffused posterior covariance.

This distinction becomes particularly important for nonlinear observation operators. In that case, an exact analogue of $A_t^\star$ would require the geometry of the diffused nonlinear posterior, which is generally unavailable in closed form and would itself be difficult to construct. EnSF avoids this requirement. As shown in Eq.~\eqref{eq:ensf_posterior_score_baseline_sec4}, it evaluates the likelihood guidance directly through the nonlinear observation operator and combines it with the diffused-prior score. The resulting score field is no longer characterized by a single global curvature matrix such as $A_t^{\mathrm{EnSF}}$, but it still defines a pseudo-time transport whose geometry evolves from the Gaussian reference toward the target posterior. Accordingly, the linear analysis above should be viewed as an explicit curvature characterization of the more general geometry-transport mechanism implemented by EnSF.
\section{Numerical experiments}
\label{sec:numerical_experiments}

This section tests the central claim of the paper, proceeding from controlled linear-Gaussian cases to a high-dimensional nonlinear benchmark. The linear-Gaussian cases are constructed so that the perturbed variational baselines sample the exact posterior once their inner optimization is solved accurately. This removes the statistical model as a possible cause of failure, so the tests directly expose the numerical limits of the variational method. Throughout, we run the practical EnSF of Section~\ref{sec:ensf_geometry_transport} rather than the exact posterior transport of Section~\ref{sec:exact_posterior_transport}, because it uses only the prior ensemble and the likelihood, without the exact posterior covariance $A^{-1}$; this way the comparison reports the performance of the practical EnSF, rather than of an exact flow-based posterior sampler that would mirror the perturbed variational baselines. The numerical results support four main findings:
\vspace{0.1cm}
\begin{itemize}[leftmargin=20pt]\itemsep0.1cm
\item \emph{The principal difficulty is numerical rather than statistical.} In the linear-Gaussian cases, perturbed variational updates recover the exact posterior when the inner problem is solved accurately. Their degradation under ill conditioning therefore reflects the computational route to the posterior, rather than a different statistical target.

\item \emph{First-level preconditioning is highly effective, but its effectiveness depends on the source of posterior anisotropy.} CVT largely removes prior-induced ill conditioning, but highly informative observations can reintroduce severe anisotropy in control space, producing the narrow stability window.

\item \emph{EnSF trades asymptotic exactness for robustness across
conditioning regimes.} A single EnSF configuration remains stable across all test cases without case-specific step-size or preconditioner adjustment. Its posterior approximation is not as accurate as a fully converged and carefully tuned exact variational sampler, but it remains effective at fixed practical budgets across large changes in posterior conditioning.

\item \emph{The observed robustness is consistent with posterior-geometry regularization and persists in the nonlinear benchmark.} The curvature paths show that the variational method repeatedly encounters a fixed endpoint geometry, whereas EnSF approaches the same endpoint through evolving intermediate geometries. In the heterogeneous Lorenz--96 experiment, the same transport route provides the most stable state tracking and lower delivered computational cost among the methods tested.
\end{itemize}
\vspace{0.2cm}

To isolate conditioning effects from nonlinear-observation effects, all experiments in this section use the linear observation model on the observed coordinates,
\begin{equation}
    y = x + E, \qquad E \sim \mathcal{N}(0,\sigma_{\rm obs}^2 I).
    \label{eq:numerical_obs}
\end{equation}
\vspace{-0.5cm}
\begin{remark}[Reproducibility]
    The numerical results presented can be reproduced using the code on \href{https://github.com/Siming-Liang/Flow-Based-Transport}{https://github.com/Siming-Liang/Flow-Based-Transport}.
\end{remark}

\subsection{Baselines, configuration, and evaluation}
\label{subsec:baselines}

We compare the flow-based update with ensemble-variational baselines that share the same background covariance information.

\emph{Three-dimensional ensemble-variational assimilation (3DEnVar).} Each ensemble member minimizes the incremental objective \eqref{eq:incremental_var} in state-increment coordinates by fixed-step gradient descent, with gradients computed by automatic differentiation, and with the background term centered at the member's own forecast.

\emph{The CVT method.} The same per-member problem is minimized in control space after the control-variable transform \eqref{eq:cvt}, with $U=L$ a Cholesky factor of the (eigenvalue-regularized) background covariance. A stable step size is available in closed form: for a selection-type observation operator, the largest eigenvalue of the CVT Hessian \eqref{eq:cvt_hessian} satisfies \cite{HabenLawlessNichols2011}
\begin{equation}
\lambda_{\max}\big(I+L^{\top}H^{\top}R^{-1}HL\big)\;\leq\;1+\lambda_{\max}(B_{n+1\mid n})\,\max_i \big(R^{-1}\big)_{ii},
\label{eq:cvt_step_bound}
\end{equation}

In the Gaussian tests below we additionally report scans over fixed step sizes to expose the stability boundary of the first-order iteration. In the Lorenz--96 experiment the CVT step is set by the adaptive step
\begin{equation}
\eta_{\mathrm{eff}}=\min\Big\{\eta_0,\;0.9\,\big/\,\big(1+\lambda_{\max}(B_{n+1\mid n})\max_i (R^{-1})_{ii}\big)\Big\}
\label{eq:cvt_adaptive_step}
\end{equation}
which is guaranteed to lie below the gradient-descent stability threshold $2/\lambda_{\max}$. The factor $0.9$ is a fine-tuned safety margin.

\emph{Perturbed-observation variants.} Each member assimilates its own perturbed observation $y_{n+1}+\epsilon^{(k)}$,
$\epsilon^{(k)}\sim\mathcal{N}(0,R_{n+1})$, with the perturbations centered over the ensemble. With identical observations, every member is driven toward the same minimizer and the analysis spread is systematically too small; with perturbed observations, the analysis ensemble reproduces exactly the linear-Gaussian posterior mean and covariance once the inner problem is solved accurately \cite{Bardsley2014RTO,Burgers1998}. The perturbed variants are therefore exact posterior samplers in the Gaussian tests below, so the comparisons measure the numerical route rather than the statistical target.

\emph{The EnSF method.} The flow-based update of Section~\ref{sec:flow_geometry} is run with the single common configuration of Section~\ref{sec:ensf_geometry_transport}: every component follows the equations there, with damping function $h(t)=1-t$, and only the reverse-step budget varies across experiments; reported iteration counts refer to the number of reverse-SDE pseudo-time discretization steps, not to the ensemble size.

\emph{Evaluation metrics.} Because the exact posterior is available analytically in the Gaussian tests, we evaluate uncertainty quantification directly through the Gaussian Kullback--Leibler (KL) divergence between the empirical posterior approximation and the exact posterior,
\begin{equation}
D_{\mathrm{KL}}\bigl(\mathcal{N}(\hat m,\hat C)\,\|\,\mathcal{N}(m,C)\bigr)
=\frac12\left[\operatorname{tr}(C^{-1}\hat C)+(m-\hat m)^\top C^{-1}(m-\hat m)-d+\log\frac{\det C}{\det \hat C}\right].
\label{eq:gaussian_kl}
\end{equation}
The exact posterior $\mathcal{N}(m^{a}_{\mathrm{ref}},C^{a}_{\mathrm{ref}})$ is available  while the empirical posterior $\mathcal{N}(\hat m^{a},\hat C^{a})$ is estimated from $10{,}000$ ensemble members. This metric penalizes both mean error and covariance error, and is therefore more informative than trajectory error alone. Note that the KL estimate has a sampling floor: two independent $10{,}000$-member samples of the same $10$-dimensional Gaussian differ by a KL of approximately $5\times10^{-3}$, so values at that level indicate agreement at the resolution of the estimator.

\subsection{Linear-Gaussian ellipse tests}
\label{subsec:gaussian_ellipse_tests}

We begin with Gaussian tests because they isolate posterior geometry from other sources of difficulty, and because the perturbed variational baselines are exact posterior samplers in this setting. What the tests measure is therefore the numerical behavior of each route: how sensitively the variational iteration depends on its step size and iteration budget under ill conditioning, and how stable the flow-based transport is under one common configuration. The prior is constructed in dimension \(d=10\) as
\begin{equation}
\Sigma = Q\Lambda Q^\top + \varepsilon I, \qquad \varepsilon=10^{-3},
\label{eq:gaussian_prior_covariance}
\end{equation}
where the first column of \(Q\) is the normalized all-ones vector $v=\mathbf{1}/\|\mathbf{1}\|_2$, and the remaining columns complete an orthonormal basis. This construction creates a dominant correlated direction and allows us to control anisotropy through the eigenvalue spectrum \(\Lambda\).

\begin{table}[!htb]
\centering
\caption{Gaussian test cases. All methods use the same prior ensemble; the perturbed variational baselines are exact posterior samplers in every case. The eigenvector associated with the largest eigenvalue is referred to as the dominant direction. The cases isolate the three conditioning regimes studied in this paper: strong prior anisotropy (Case~2), weakly observed directions (Case~3), and highly informative observations (Case~4); Case~1 is a well-conditioned control.}
\label{tab:gaussian_cases}
\small
\begin{tabular}{c l l l}
\toprule
\textbf{Case} & \textbf{Prior geometry} & \textbf{Observed coordinates} & \textbf{Observation std} \\
\midrule
1 & Well-conditioned, $\kappa \approx 10$ & dominant ($10\%$) & $\sigma_{\rm obs}=0.1$ \\
2 & Ill-conditioned, $\kappa \approx 1.5 \times 10^8$ & all ($100\%$) & $\sigma_{\rm obs}=50$ \\
3 & Same as Case 2 & dominant ($10\%$) & $\sigma_{\rm obs}=5$ \\
4 & Same as Case 2 & dominant ($10\%$) & $\sigma_{\rm obs}=0.1$ \\
\bottomrule
\end{tabular}
\end{table}

We consider four cases, summarized in Table~\ref{tab:gaussian_cases}. Case~1 is moderately anisotropic and well conditioned: the eigenvalues are approximately $\{10,1, \ldots, 1\}$ with condition number \(\kappa\approx 10\), and only the dominant direction is observed. Case~2 is strongly anisotropic and ill conditioned: the eigenvalues are approximately $\{10^6,\;10^5, \ldots, 10^{-2},\; 10^{-3}\}$ with condition number \(\kappa\approx 1.5\times 10^8\), under full observation. Case~3 uses the same ill-conditioned prior, but only the dominant component is observed. Case~4 keeps the sparse observation of Case~3 and reduces the observation standard deviation to $\sigma_{\rm obs}=0.1$, which makes the observed direction extremely precise relative to the prior spread.

\paragraph{\bf Case 1: moderate anisotropy with partial observation}
Case~1 serves as a control case. The prior is only mildly anisotropic, so the likelihood constraint $y=2$ with $\sigma_{\rm obs}=0.1$ produces a clear posterior target rather than a narrow ill-conditioned valley: the observation sharply constrains the observed coordinate, while the prior correlation determines the conditional spread in the unobserved coordinate. Figure~\ref{fig:case1_results} compares the reference posterior with EnSF, unperturbed 3DEnVar, and perturbed 3DEnVar. The distinction that matters here is not speed but spread: without perturbed observations, every member is pulled toward the same minimizer, so the analysis ensemble collapses in the observed direction and cannot represent the posterior uncertainty even in this benign geometry. Adding observation perturbations restores the correct spread.

Table~\ref{tab:case1_kl} reports the KL values. At practical budgets the methods are comparable: EnSF reaches KL $0.28$ within $100$ steps and is insensitive to the budget, while perturbed 3DEnVar reaches its minimum KL $0.17$ at convergence ($383$ iterations at the stable step $\eta=0.01$). The residual KL of the perturbed method reflects the convergence threshold ($10^{-6}$ on the cost change) rather than a statistical limitation: as an exact sampler, it approaches the sampling floor of Section~\ref{subsec:baselines} when the inner problem is solved to higher accuracy. The takeaway from Case~1 is that the variational method is fully adequate in well-conditioned geometry provided perturbed observations are used, and that EnSF is competitive there without any case-specific adjustment.

\begin{figure}[!htb]
    \centering
    \includegraphics[width=0.85\linewidth]{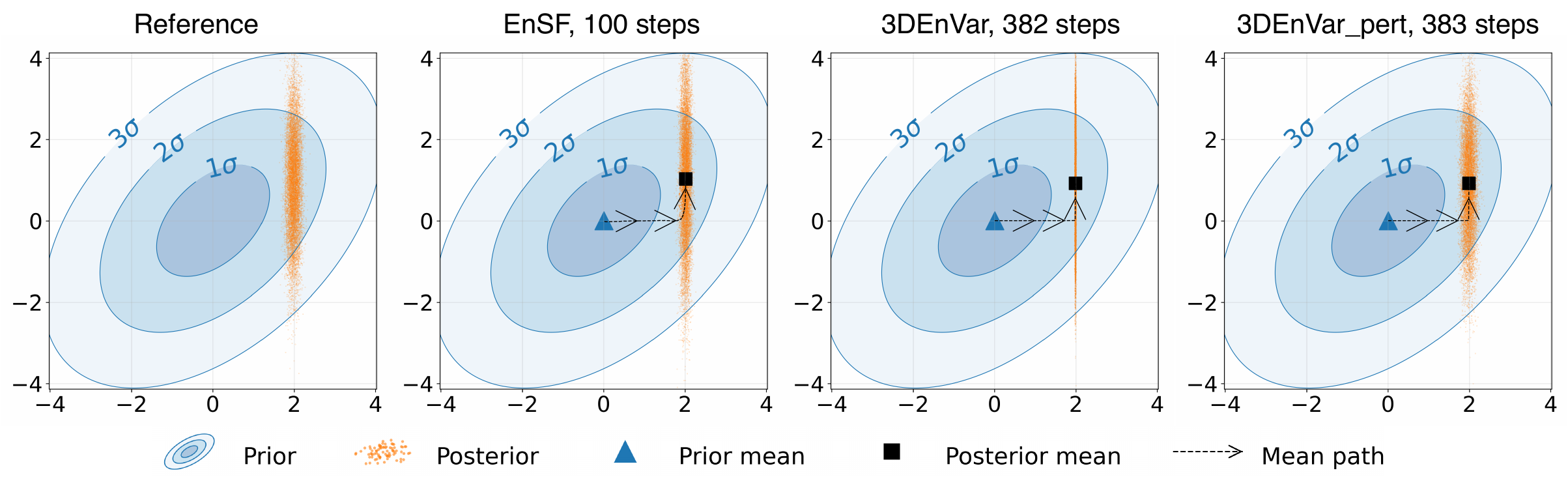}
    \caption{Case~1: well-conditioned control case, shown in the plane of the observed dominant direction (horizontal) and an unobserved direction (vertical). Without perturbed observations the analysis collapses in the observed direction even in this benign geometry; with perturbations the variational method is fully adequate, and EnSF matches it with the same common configuration used in all cases.}
    \label{fig:case1_results}
\end{figure}

\begin{table}[!htb]
\centering
\caption{Case~1: KL divergence at fixed iteration budgets. Step sizes are close to the largest stable values; the variational runs stop when the cost change falls below $10^{-6}$. At this level of conditioning the two routes are comparable: perturbed 3DEnVar attains the lower KL at convergence, while EnSF is insensitive to its budget.}
\label{tab:case1_kl}
\small
\begin{tabular}{cccccc}
\toprule
\textbf{Method} & \textbf{Step size $\eta$} & \textbf{20 steps} & \textbf{50 steps} & \textbf{100 steps} & \textbf{Min KL (iter.)} \\
\midrule
3DEnVar & $0.01$ & $0.78$ & $0.69$ & $0.57$ & $0.22$ ($382$) \\
3DEnVar (perturbed) & $0.01$ & $0.76$ & $0.67$ & $0.55$ & $0.17$ ($383$) \\
EnSF & --- & $0.33$ & $0.31$ & $0.28$ & $0.26$ ($300$) \\
\bottomrule
\end{tabular}
\end{table}

\paragraph{\bf Case 2: extreme anisotropy with full observation}
Case~2 changes the conditioning, and the observation operator now sees all coordinates. Case~2 is not difficult because observations are missing; it is difficult because the correct posterior uncertainty must be reconstructed inside a severely ill-conditioned geometry. The prior ensemble is highly elongated and concentrated near a strongly correlated direction, with condition number $\kappa(B)\approx 1.5\times 10^8$. In this setting, the use of full observations makes the posterior center clearly identifiable, so locating the posterior mean is not the main challenge. To make the posterior covariance visible in the two-dimensional projection, we set the observation vector to $\mathbf{y}=600\cdot\mathbf{1}_{10}$ and use a relatively large observation standard deviation, $\sigma_{\rm obs}=50$. If the observation noise were much smaller, the posterior would collapse visually to an almost point mass, obscuring the comparison of posterior spread. With $\sigma_{\rm obs}=50$, the posterior still has a visible elliptical shape, so the experiment can test whether a method recovers not only the posterior mean but also the residual uncertainty.

Figure~\ref{fig:case2_all} and Table~\ref{tab:case2_kl} summarize the results, and three observations capture the case. First, state-space gradient descent is impractically slow: at the stable step $\eta=0.006$, 3DEnVar reaches its convergence threshold only after about $1.6\times10^{6}$ iterations (stopping at iteration $1{,}629{,}398$), and its final KL of $3.31$ still reflects a too-narrow spread; at $\eta=0.007$ the iteration diverges. Second, the CVT repairs the speed but not the spread: at $\eta=0.004$ it converges within about $110$ iterations, but without perturbations the analysis collapses and the KL stalls near $5$; with perturbed observations the CVT reaches KL $0.005$ --- the sampling floor --- within $130$ iterations, consistent with its exact-sampler property. Third, the flow-based update is stable and accurate without any adjustment: EnSF reaches KL $0.12$ at $20$ steps and $0.07$ at $200$ steps. The case also exposes the step-size cliff of the first-order variational method: for both CVT variants, $\eta=0.004$ converges within roughly a hundred iterations while $\eta=0.005$ diverges. The stable window must be found per problem; the adaptive bound \eqref{eq:cvt_adaptive_step} provides a safe step automatically, at the price of conservativeness.

\begin{figure}[!htb]
    \centering    \includegraphics[width=0.75\linewidth]{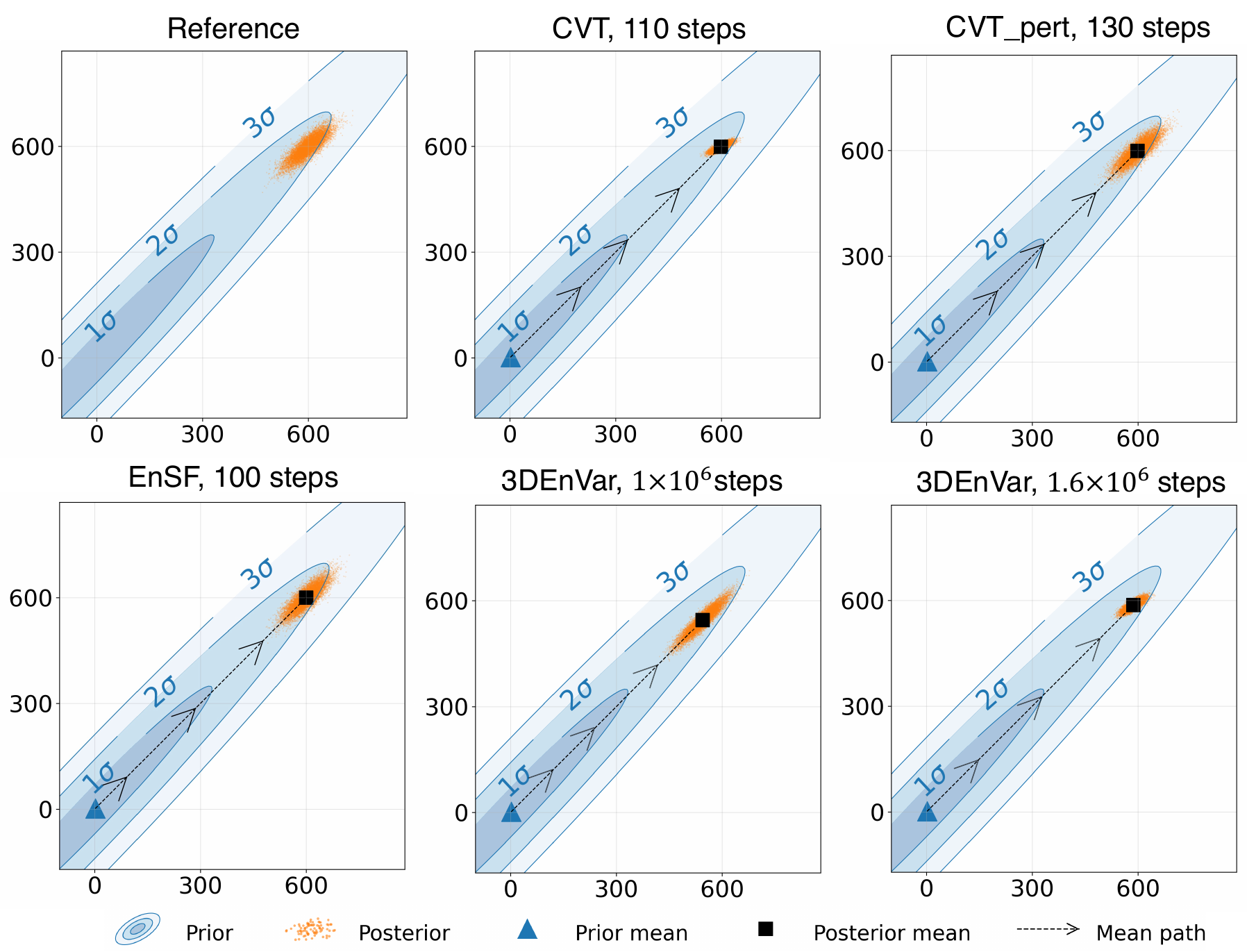}
    \caption{Case~2: strong prior anisotropy ($\kappa(B)\approx1.5\times10^{8}$) under full observation, projected onto the plane of the two leading prior directions. Unperturbed CVT collapses the spread; perturbed CVT matches the reference; EnSF tracks the reference with $100$ steps; 3DEnVar remains too narrow even at convergence ($\approx1.6\times10^{6}$ iterations) --- under severe ill conditioning it is the state-space iteration that fails, not the statistical model.}
    \label{fig:case2_all}
\end{figure}

\begin{table}[!htb]
\centering
\caption{Case~2 (full observation): KL divergence at fixed iteration budgets. ``$\infty$'' denotes divergence of the iteration. The perturbed CVT at $\eta=0.004$ reaches the sampling floor. Preconditioning restores the speed that ill conditioning takes from the state-space iteration; EnSF reaches comparable accuracy with $100$ steps and no per-case tuning.}
\label{tab:case2_kl}
\small
\begin{tabular}{cccccc}
\toprule
\textbf{Method} & \textbf{Step size $\eta$} & \textbf{20 steps} & \textbf{50 steps} & \textbf{100 steps} & \textbf{Min KL (iter.)} \\
\midrule
3DEnVar & $0.007$ & $1236$ & $1236$ & $1236$ & $\infty$ (diverges) \\
3DEnVar & $0.006$ & $1236$ & $1236$ & $1235$ & $3.31$ ($1.6\times10^{6}$) \\
CVT & $0.005$ & $1453$ & $1964$ & $3244$ & $\infty$ (diverges) \\
CVT & $0.004$ & $3.16$ & $4.13$ & $5.08$ & $5.07$ ($110$) \\
CVT (perturbed) & $0.005$ & $1454$ & $1965$ & $3246$ & $\infty$ (diverges) \\
CVT (perturbed) & $0.004$ & $0.35$ & $0.08$ & $0.009$ & $0.005$ ($130$) \\
EnSF & --- & $0.12$ & $0.11$ & $0.09$ & $0.07$ ($200$) \\
\bottomrule
\end{tabular}
\end{table}

\paragraph{\bf Case 3: extreme anisotropy with partial observation}
In Case~3, the prior distribution is the same severely ill-conditioned Gaussian used in Case~2, with condition number $\kappa(B)\approx 1.5\times 10^8$. However, instead of observing all state components, we only observe the dominant direction. Even though the observation is precise in the dominant observed direction, the remaining state components must be inferred through the highly ill-conditioned prior covariance. Thus, the posterior target is clear in the observed coordinate but still difficult in the unobserved directions. The observation value is still set to $y=600$, but the observation standard deviation is reduced to $\sigma_{\rm obs}=5$. This small observation noise tightly constrains the observed coordinate, so the posterior samples concentrate near the observed value; in the two-dimensional projection this produces a narrow band rather than the broader posterior ellipse of Case~2. Because the remaining directions are not directly observed, the posterior does not collapse to a point: the uncertainty in the unobserved components is determined indirectly by the conditional structure of the ill-conditioned prior. This case therefore tests whether a method can propagate precise information from the observed direction into the unobserved directions without destroying the posterior covariance.

\begin{figure}[!htb]
    \centering
    \includegraphics[width=0.9\linewidth]{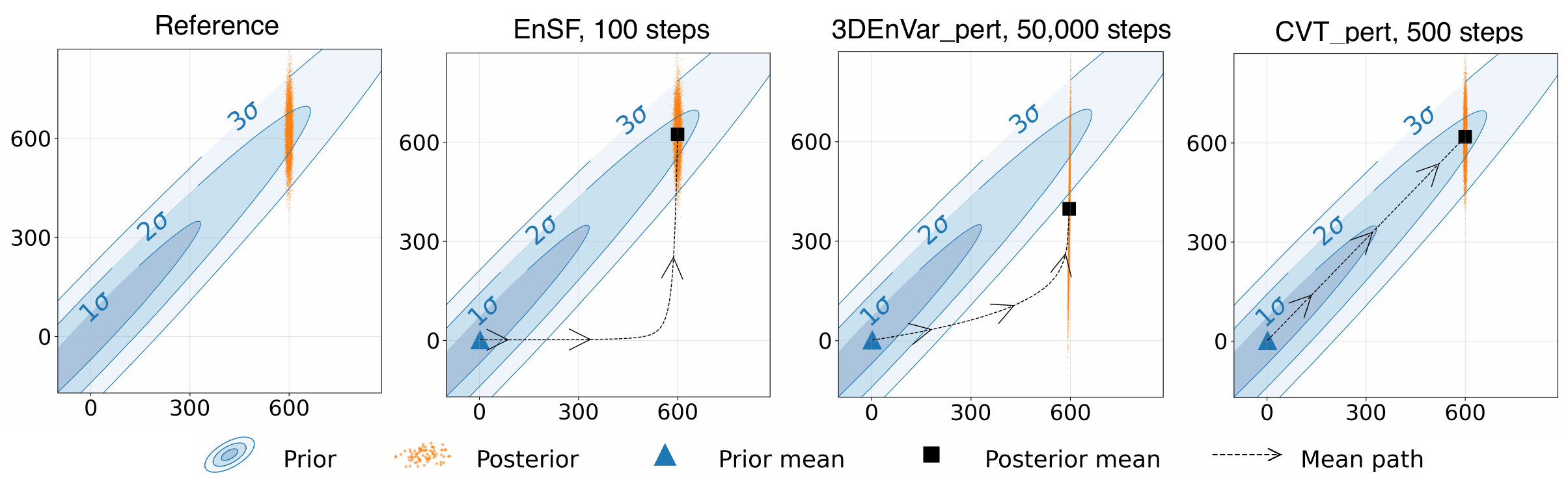}
    \caption{Case~3: same ill-conditioned prior with weakly observed directions: only the dominant component observed. EnSF and the perturbed CVT recover both the posterior mean and the anisotropic conditional spread; the perturbed 3DEnVar fits the observed coordinate but misplaces the mean and distorts the spread --- at realistic budgets the ill-conditioned state-space iteration cannot propagate observation information into unobserved directions.}
    \label{fig:case3_results}
\end{figure}
\begin{table}[!htb]
\centering
\caption{Case~3 (sparse observation): KL divergence at fixed iteration budgets. The state-space iteration remains far from the posterior at every budget shown, while the perturbed CVT and EnSF both recover it; EnSF does so with the common configuration.}
\label{tab:case3_kl}
\small
\begin{tabular}{cccccc}
\toprule
\textbf{Method} & \textbf{Step size $\eta$} & \textbf{50 steps} & \textbf{100 steps} & \textbf{500 steps} & \textbf{50{,}000 steps} \\
\midrule
3DEnVar (perturbed) & $0.005$ & --- & --- & --- & $99$ \\
CVT (perturbed) & $9.1\times10^{-5}$ & $196$ & $22$ & $0.007$ & --- \\
EnSF & --- & $0.11$ & $0.09$ & $0.06$ & --- \\
\bottomrule
\end{tabular}
\end{table}
Figure~\ref{fig:case3_results} and Table~\ref{tab:case3_kl} summarize the results. EnSF transports the ensemble toward the reference posterior in both coordinates: the mean is recovered in the observed and the unobserved coordinate, the uncertainty is sharply reduced in the observed direction, and the conditional spread in the unobserved direction is retained. The perturbed CVT again behaves as designed, reaching KL $0.007$ within $500$ iterations at its stable step ($\eta=9.1\times10^{-5}$), although at small budgets it is still far from the target (KL $196$ at $50$ and $22$ at $100$ iterations); EnSF is more accurate at those budgets (KL $0.11$ and $0.09$) and insensitive to the budget. The perturbed state-space 3DEnVar fails in this regime: after $50{,}000$ iterations its KL is still $99$, because the observation information is fitted in the observed coordinate but not propagated correctly through the ill-conditioned prior covariance into the unobserved components, and the resulting spread is distorted. Case~3 is the right regime to separate \emph{mean fit} from \emph{posterior geometry}: a method can align its mean with the observed component and still misrepresent the posterior in the unobserved directions.

\paragraph{\bf Case 4: extreme anisotropy with a precise sparse observation}
Case~4 keeps the prior and the sparse observation pattern of Case~3 but reduces the observation standard deviation to $\sigma_{\rm obs}=0.1$. This makes the observed direction extremely precise relative to the prior spread, and it is the regime in which first-level preconditioning runs out of headroom. For a single observed coordinate, the CVT Hessian \eqref{eq:cvt_hessian} is a rank-one update of the identity with largest eigenvalue $1+B_{00}/\sigma_{\rm obs}^{2}$, where $B_{00}$ is the prior variance of the observed coordinate; in this case that eigenvalue is of order $10^{7}$, so the CVT Hessian itself is severely ill conditioned despite the transform (see Section~\ref{subsec:conditioning_transport}). This is precisely the regime for which second-level preconditioning was developed in the variational literature (Section~\ref{sec:classical_preconditioning}).
\begin{figure}[!htb]
    \centering
    \includegraphics[width=0.75\linewidth]{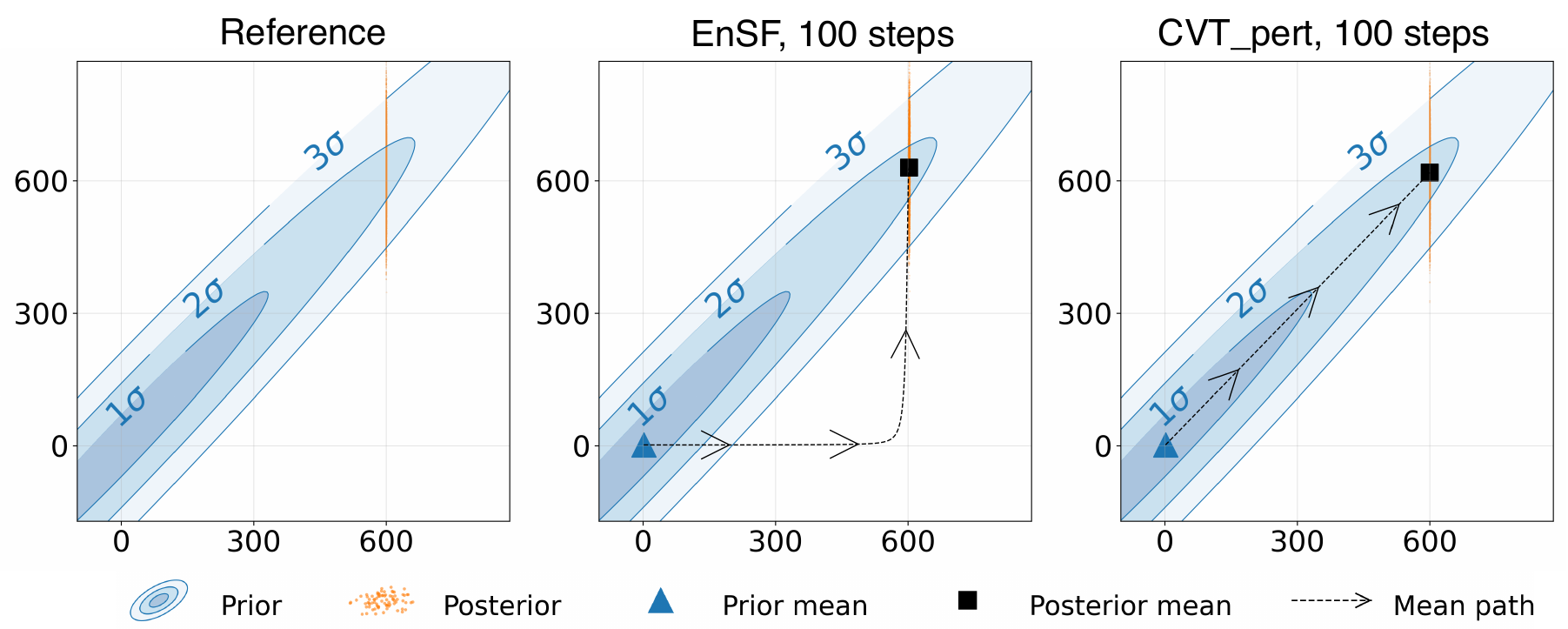}
    \caption{Case~4: precise sparse observation ($\sigma_{\rm obs}=0.1$; the perturbed CVT uses $\eta=10^{-7}$). Both recover the narrow posterior band and its conditional spread; the EnSF mean path first travels along the observed coordinate and then propagates the information into the unobserved direction through the prior correlation. This is the regime where first-level preconditioning reaches its limit ($\kappa\approx10^{7}$): the CVT succeeds only inside a step window below $2\times10^{-7}$, while EnSF runs unchanged.}
    \label{fig:case4}
\end{figure}

\begin{table}[!htb]
\centering
\caption{Case~4 (precise sparse observation): KL divergence of the perturbed CVT across step sizes, against EnSF with the common configuration. ``$\infty$'' denotes divergence. The stable window of the first-order iteration is extremely narrow, and the best small-budget accuracy occurs at an interior step size near the stability boundary. Tuned inside that window the variational method is exact, while EnSF requires no such tuning.}
\label{tab:case4_kl}
\small
\begin{tabular}{cccccc}
\toprule
\textbf{Method} & \textbf{Step size $\eta$} & \textbf{20 steps} & \textbf{50 steps} & \textbf{100 steps} & \textbf{1000 steps} \\
\midrule
CVT (perturbed) & $1.9\times10^{-7}$ & $1.4\times10^{6}$ & $\infty$ & $\infty$ & $\infty$ \\
CVT (perturbed) & $1.8\times10^{-7}$ & $2.9\times10^{2}$ & $15$ & $0.1$ & $0.008$ \\
CVT (perturbed) & $1.0\times10^{-7}$ & $0.008$ & $0.008$ & $0.008$ & $0.008$ \\
CVT (perturbed) & $1.0\times10^{-8}$ & $1036$ & $676$ & $225$ & $0.008$ \\
CVT (perturbed) & $1.0\times10^{-9}$ & $1937$ & $1814$ & $1627$ & $228$ \\
EnSF & --- & $0.91$ & $0.46$ & $0.36$ & $0.33$ \\
\bottomrule
\end{tabular}
\end{table}

The consequences for the first-order iteration are shown in Table~\ref{tab:case4_kl}. The stable step-size window becomes razor thin: $\eta=1.8\times10^{-7}$ converges while $\eta=1.9\times10^{-7}$ diverges, a margin of about five percent. Below the window the iteration stalls: at $\eta=10^{-9}$ the KL is still $228$ after $1{,}000$ iterations. Moreover, the best accuracy at a small fixed budget is achieved neither at the largest stable step nor at a conservative one, but at an interior value close to the stability boundary: at $\eta=10^{-7}$ the perturbed CVT reaches KL $0.008$ at every reported budget, whereas the largest stable step needs hundreds of iterations to recover from its initial transient. Selecting such a step requires per-problem tuning against a divergence cliff. The flow-based update, run unchanged with the common configuration, remains stable with KL $0.91$, $0.46$, $0.36$, and $0.33$ at $20$, $50$, $100$, and $1{,}000$ steps: less accurate than a perfectly tuned exact sampler, but obtained without any tuning in a regime where the tuned method sits a few percent away from divergence. Figure~\ref{fig:case4} confirms that both EnSF and the perturbed CVT recover the posterior structure; the EnSF mean path first moves along the observed coordinate and then spreads the information into the unobserved direction through the prior correlation.

\subsection{Conditioning along the transport path}
\label{subsec:conditioning_transport}

Figure~\ref{fig:cond} makes the posterior-geometry regularization of Proposition~\ref{prop:ensf_curvature} visible in the four Gaussian cases and summarizes the conditioning that each numerical route actually faces. The solid curve shows the condition number of the EnSF curvature $A_t^{\mathrm{EnSF}}=\Sigma_t^{-1}+h(t)H^{\top}R^{-1}H$ of Eq.~\eqref{eq:ensf_curvature_sec33}, evaluated along the experimental schedule: the reverse transport starts at the well-conditioned reference end ($\kappa\to1$ as $t\to1$) and steepens toward the posterior curvature only near the analysis end, so the flow spends most of the transport on well-conditioned intermediate geometry and never solves a fixed ill-conditioned system. Because the reverse integration terminates at a small regularized endpoint rather than exactly at $t=0$, the terminal value remains somewhat below $\kappa(A)$ in the severely ill-conditioned cases. The horizontal lines show the fixed condition numbers of the objects the optimization route works with at every iteration: the background covariance, the state-space Hessian $B^{-1}+H^{\top}R^{-1}H$, and the CVT Hessian $I+L^{\top}H^{\top}R^{-1}HL$.

Two readings are worth separating. First, in Case~1 the Hessian condition numbers exceed $\kappa(B)\approx10$ because the precise observation itself adds curvature: the conditioning of the update is a property of the full posterior geometry, not of the prior alone. Second, the CVT line shows both the strength and the limit of first-level preconditioning. In Cases~2 and~3 it reduces the conditioning from $\kappa(B)\approx1.5\times10^{8}$ to about $4\times10^{2}$ and $4.5\times10^{3}$, respectively. Indeed, in these two cases the CVT Hessian is better conditioned than most of the transport path itself: where first-level preconditioning applies, it is excellent. In Case~4, however, the precise sparse observation creates the stiff mode discussed above, and the CVT Hessian remains at $\kappa\approx10^{7}$: whenever $\lambda_{\max}(B)\max_i(R^{-1})_{ii}$ is large, first-level preconditioning is not sufficient, and the optimization route requires second-level preconditioning of the transformed observation term (Section~\ref{sec:classical_preconditioning}). The flow-based route requires no algorithmic change across this regime shift: the damped likelihood term raises the curvature it encounters (compare Cases~2 and~4), but the transport meets the data-informed geometry only near the analysis end, where the optimization route faces it at every iteration.

\begin{figure}[!ht]
    \centering
    \includegraphics[width=0.9\linewidth]{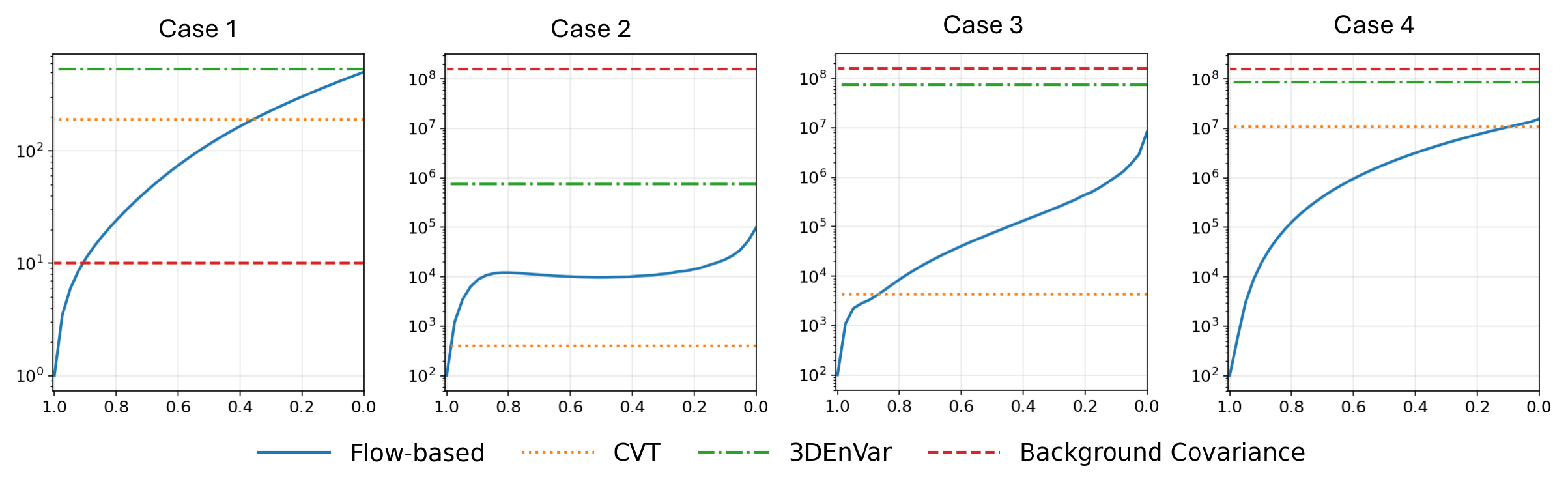}
    \caption{Condition numbers in the four Gaussian cases. Solid curve: the EnSF curvature conditioning $\kappa(A_t^{\mathrm{EnSF}})=\kappa(\Sigma_t^{-1}+h(t)H^{\top}R^{-1}H)$ (Proposition~\ref{prop:ensf_curvature}), evaluated along the experimental schedule from the reference end ($t=1$) to the analysis end ($t=0$). Horizontal lines: condition numbers of the background covariance, the 3DEnVar Hessian, and the CVT Hessian. First-level preconditioning is effective in Cases~2--3 but leaves $\kappa\approx10^{7}$ in Case~4. This is the posterior-geometry regularization mechanism: the optimization route faces its fixed conditioning at every iteration, whereas the transport starts isotropic and meets the data-informed geometry only at the analysis end.}
    \label{fig:cond}
\end{figure}

\subsection{Lorenz--96 multiscale example under sparse observations}
\label{subsec:l96_multiscale}

The nonlinear dynamical benchmark in this study is a heterogeneous Lorenz--96 system \cite{LorenzEmanuel1998}. Let \(d=10000\),
\(x(t)=(x_1(t),\ldots,x_d(t))^\top\in\mathbb{R}^d\), and impose periodic indexing.
We use the scaled construction
\begin{equation}
\frac{d x_i}{dt}=s_i (x_{i+1}-x_{i-2})x_{i-1} - x_i + F_i,\qquad i=0,\ldots,d-1,
\label{eq:heterogeneous_l96}
\end{equation}
where the coordinate-dependent scale and forcing vectors are formed by repeating the four-component blocks
\begin{equation}
(\bar s_1,\bar s_2,\bar s_3,\bar s_4)=(600,\,200,\,0.5,\,0.15),\qquad
(\bar F_1,\bar F_2,\bar F_3,\bar F_4)=(2,\,10,\,200,\,500).
\label{eq:hetero-l96}
\end{equation}
The four-component blocks in~\eqref{eq:hetero-l96} are repeated 2500 times. This construction introduces strong heterogeneity across coordinates and creates a challenging multiscale state with widely different effective nonlinear couplings and time scales. Compared with the classical Lorenz--96 model, the coordinate-dependent scaling $s_i$ and forcing $F_i$ introduce pronounced heterogeneity across state variables: the nonlinear interaction term is multiplied by coefficients $s_i$ ranging from $1.5\times10^{-1}$ to $6\times10^{2}$, introducing differences of more than three orders of magnitude in the effective strength of the nonlinear coupling. This heterogeneity is visible in the trajectory projections in Figure~\ref{fig:l96_trajectory}, where different coordinates evolve on markedly different numerical scales. We therefore use this heterogeneous Lorenz--96 setting as a nonlinear dynamical stress test for conditioning-sensitive assimilation.

In our numerical experiments we integrate \eqref{eq:hetero-l96} using a fourth-order Runge--Kutta method with a small time step ($\Delta t=10^{-5}$) to account for the stiffness induced by the large scaling coefficients. The ensemble size is $1{,}000$ and the observation model is the same as Eq.~\eqref{eq:numerical_obs}. We observe the indices $i=0,4,8,\ldots$ using zero-based indexing, corresponding to the components with scale factor $s_i=600$, so $25\%$ of the state is observed with $\sigma_{\rm obs}=1$. We compare three methods: EnSF, the perturbed CVT, and the perturbed 3DEnVar. All three construct empirical covariance matrices from their respective ensembles, and localization is applied to these ensemble covariance matrices using the same Gaspari--Cohn taper \cite{gaspari1999construction} with radius $r_{\rm loc}=4$, so that any difference in performance is not due to different localization choices. Both variational baselines use perturbed observations. The CVT step size is set by the adaptive bound \eqref{eq:cvt_adaptive_step}, which removes the per-cycle step-size tuning; the safety factor $0.9$ is still required to avoid failures. The state-space 3DEnVar uses $\eta=10^{-4}$, where a slightly larger value causes loss divergence. The variational baselines iterate until the cost change falls below $10^{-3}$ (the perturbed CVT requires about $9{,}000$ iterations per assimilation on average); EnSF uses $1{,}000$ reverse-SDE steps. The assimilation gap is $\tau_{\mathrm{DA}}=10$ model steps, i.e., $10^{-4}$ time units. Because the state amplitudes vary by orders of magnitude across coordinates, the ensemble-mean error is reported as a relative root-mean-square error (RMSE), normalized componentwise by the temporal standard deviation of the reference trajectory:
\begin{equation}
\mathrm{RMSE}_n=\left(\frac{1}{d}\sum_{i=1}^{d}\left(\frac{\bar x_{n,i}-x^{\mathrm{ref}}_{n,i}}{s^{\mathrm{ref}}_i}\right)^{2}\right)^{1/2},
\qquad s^{\mathrm{ref}}_i=\operatorname{std}_n\big(x^{\mathrm{ref}}_{n,i}\big),
\label{eq:l96_rmse}
\end{equation}
where \(\bar x_n\) is the ensemble mean and \(x_n^{\mathrm{ref}}\) is the reference state; the normalization prevents the largest-amplitude coordinates from dominating the metric.

\paragraph{Remark} In ensemble-variational DA, localization is already known to play a dual role as statistical regularization \cite{Tong2018LocalEnKF,MorzfeldHodyss2023} and as a modification of optimization geometry. Localization has also been extended to nonlinear particle filters \cite{Poterjoy2016LPF}. In flow-based methods, analogous questions remain to be analyzed carefully: how localization should enter score estimation, how it changes the transported geometry, and whether it helps or distorts the near-isotropic reference-space interpretation in strongly heterogeneous systems. Since localization is not the main focus of this paper, we do not attempt a systematic study here. Instead, for the Lorenz--96 experiment, we tested localization radii $r_{\rm loc}=1,\ldots,10$ and selected $r_{\rm loc}=4$, which gave the best performance for the variational baselines. This choice provides a favorable and fair baseline for the variational method in the comparison below.

\begin{figure}[!htb]
    \centering
    \includegraphics[width=0.5\linewidth]{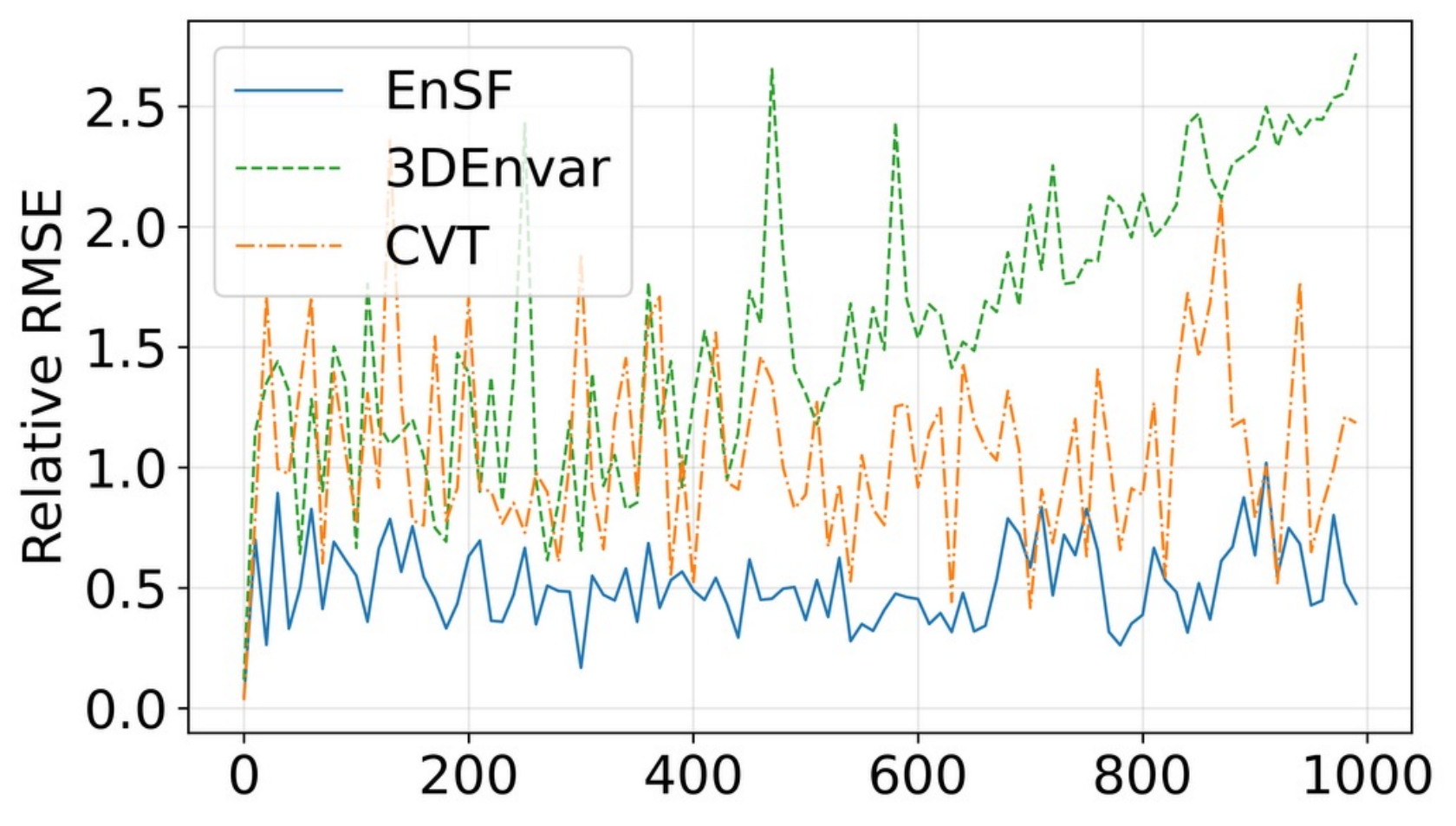}
    \caption{Lorenz--96 relative RMSE \eqref{eq:l96_rmse} at $\tau_{\mathrm{DA}}=10$ for EnSF, the perturbed CVT with the adaptive step \eqref{eq:cvt_adaptive_step}, and the perturbed 3DEnVar. All three methods share the same localized ensemble covariances and perturbed observations. EnSF remains near $0.5$ throughout; the CVT oscillates between about $1$ and $1.5$ with intermittent spikes; the state-space 3DEnVar drifts away from the truth. Under identical covariance information, the transport route is the most stable of the three in a nonlinear, strongly heterogeneous system.}
    \label{fig:l96_rmse_all}
\end{figure}

Figure~\ref{fig:l96_rmse_all} shows the relative RMSE over $1{,}000$ model steps ($100$ assimilation cycles). EnSF remains stable around $0.5$ for the whole window. The perturbed CVT tracks the truth but oscillates between about $1$ and $1.5$, with intermittent spikes above $2$. The perturbed state-space 3DEnVar drifts away, with the error growing beyond $2.5$ by the end of the window. The ranking matches the Gaussian tests: the closer the route stays to well-conditioned geometry, the more robust the result, even though all three methods share the same covariance information and the same perturbed-observation construction.

Trajectory projections are shown in Figure~\ref{fig:l96_trajectory}: dimension $0$ (observed) against dimensions $1$ and $2$ (unobserved). EnSF preserves the attractor structure; the perturbed CVT stays on the attractor with occasional departures; the perturbed 3DEnVar trajectory departs substantially, consistent with the RMSE histories.

\begin{figure}[!ht]
    \centering
    \includegraphics[width=0.8\linewidth]{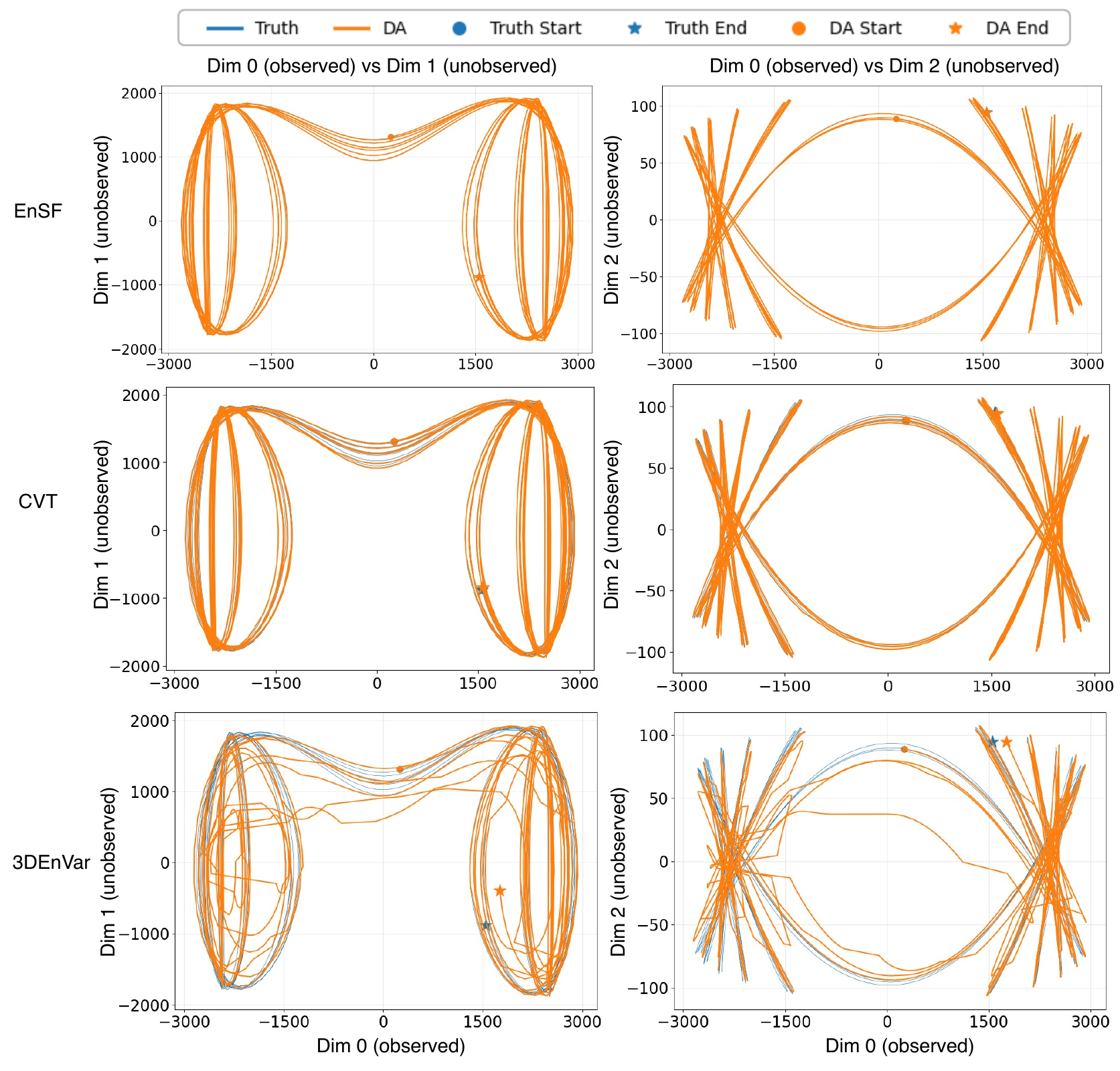}
    \caption{Lorenz--96 trajectory projections at $\tau_{\mathrm{DA}}=10$, for dimension $0$ (observed) against dimensions $1$ and $2$ (unobserved). Rows: EnSF, perturbed CVT, perturbed 3DEnVar. EnSF preserves the attractor structure, the CVT shows occasional misses, and the 3DEnVar trajectory departs substantially --- the robustness seen in the Gaussian tests persists in tracking a high-dimensional chaotic system.}
    \label{fig:l96_trajectory}
\end{figure}

The Lorenz--96 test is consistent with the main message of the Gaussian benchmarks in a nonlinear dynamical setting. The Gaussian tests provide direct posterior-distribution diagnostics through the KL divergence, whereas the Lorenz--96 test evaluates dynamical tracking through ensemble-mean RMSE and trajectory projections; under identical covariance information, the flow-based update is the most stable of the three routes.

\subsection{Computational cost}
\label{subsec:cost}

We now compare the computational cost of the two routes at $d=10^{4}$. Per iteration and per ensemble member, the dominant work of a CVT gradient step is the pair of dense products with the covariance square root, about $4d^{2}$ operations, while an EnSF reverse-SDE step requires a solve with $\alpha_t^{2}\Sigma+\beta_t^{2}I$ for the prior score \eqref{eq:sec5_diffused_score}. For a localized covariance, which is banded with periodic boundary, this solve can be performed exactly by a blocked banded Cholesky sweep with a Schur-complement correction for the periodic border; with block size $b_{\mathrm{size}}$ the cost is about $(12\,b_{\mathrm{size}}+15)\,d$ operations per member, which at $d=10^{4}$ and $b_{\mathrm{size}}=100$ is roughly $30\times$ fewer operations than a CVT iteration. In practice, however, GPU limits reduce this due to the banded sweep is bound by memory bandwidth and launch latency rather than arithmetic.

Figure~\ref{fig:l96_cost} reports measured wall-clock times at a fixed budget of $1{,}000$ iterations or reverse-SDE steps per assimilation and $N=1{,}000$ members. All timings on a single NVIDIA RTX A5000 GPU in single precision, with GPU-synchronized timers, warm-up excluded, and the model integration timed separately. The dashed lines in panel (a) show the theoretical times obtained from the operation counts; the measured times follow the theoretical slopes closely, with a roughly constant offset from launch latency and memory traffic. Panel (b) compares fixed budgets with the budgets actually needed for the results of Figure~\ref{fig:l96_trajectory}: since the perturbed CVT requires about $9{,}000$ iterations per assimilation to converge, its delivered cost is approximately $275$\,s per assimilation, against $20.5$\,s for EnSF with $1{,}000$ steps. At delivered accuracy, the flow-based route is thus roughly an order of magnitude cheaper in this experiment, while also being the more accurate method in Figure~\ref{fig:l96_rmse_all}.

\begin{figure}[!ht]
    \centering
    \includegraphics[width=0.85\linewidth]{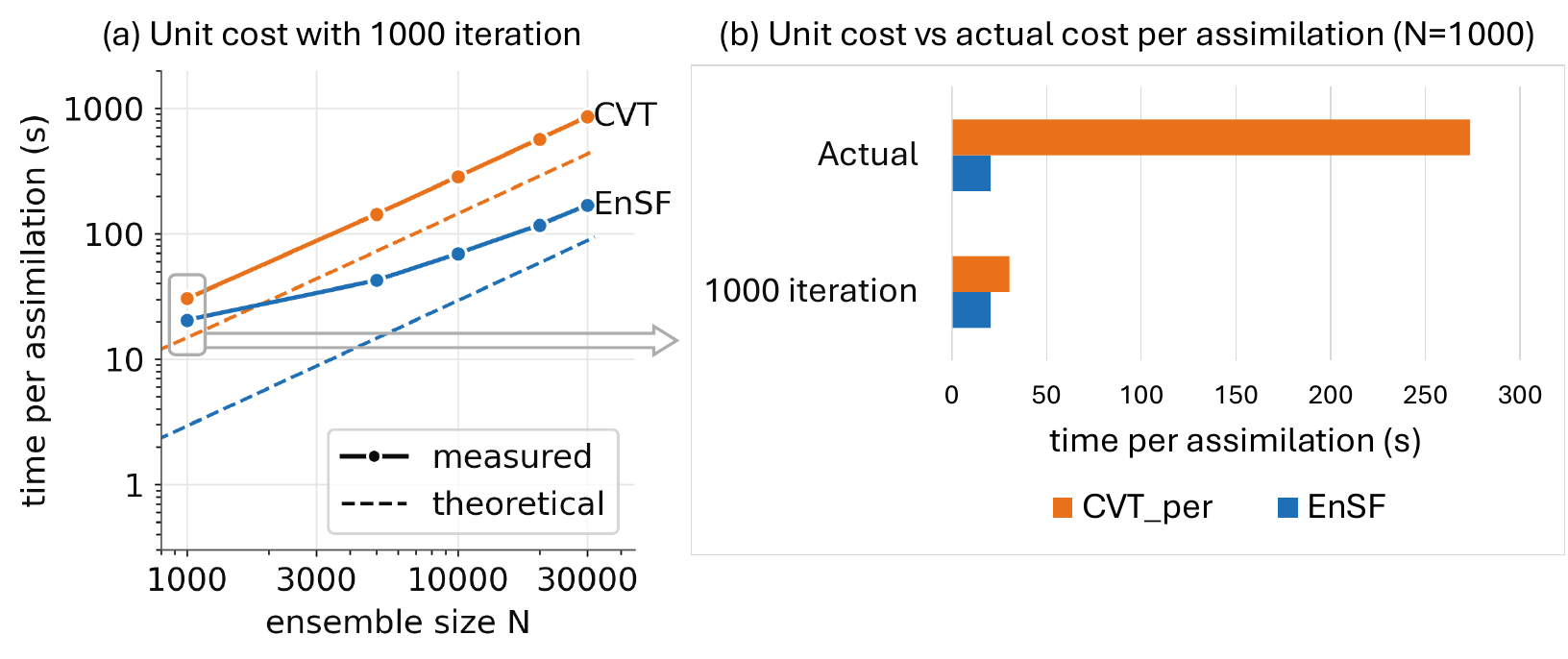}
    \caption{Computational cost at $d=10^{4}$. (a) Measured time per assimilation versus ensemble size at a fixed budget of $1{,}000$ iterations/steps (solid, markers), against the theoretical time at the GPU's nominal single-precision peak (dashed). (b) Fixed-budget versus delivered cost at $N=1{,}000$: the perturbed CVT needs about $9{,}000$ iterations per assimilation to produce the results of Figure~\ref{fig:l96_trajectory} ($\approx275$\,s), while EnSF delivers its result with $1{,}000$ steps ($20.5$\,s). At delivered accuracy the flow-based route is roughly an order of magnitude cheaper.}
    \label{fig:l96_cost}
\end{figure}

\subsection{Discussion}
\label{sec:discussion}

The experiments support a consistent reading. First, the Gaussian tests do not show that variational methods fail statistically. On the contrary, the perturbed-observation baselines are exact posterior samplers in this setting, and when their step size is tuned and the inner problem is iterated to convergence they are the most accurate methods in the comparison: the perturbed CVT reaches the sampling floor of the KL estimator in Cases~2--4. What the tests isolate is the numerical behavior of the optimization route. The state-space iteration is impractically slow (Case~2) or fails to propagate information into unobserved directions at realistic budgets (Case~3); the unperturbed variants systematically collapse the analysis spread (all cases); and the stable step-size window narrows to a few percent around a divergence cliff exactly when the observation is precise relative to the prior spread (Case~4). Each of these failure modes has an established remedy within the variational framework, i.e., perturbed observations, the control-variable transform, and second-level preconditioning, but the remedies are regime-dependent choices that must be diagnosed and re-tuned, whereas the flow-based update ran in every regime with one common configuration.

Second, the comparison is about routes, not solvers. Gradient descent, conjugate-gradient methods, the CVT, and Lanczos-based second-level preconditioners are different ways of solving an equivalent fixed local optimization problem, and a well-preconditioned Krylov method would solve the quadratic inner problems of the Gaussian cases quickly (Section~\ref{sec:classical_preconditioning}); we make no claim that the flow-based update is faster or more accurate than such a solver on a given quadratic. The structural difference is that the flow-based update never solves a fixed ill-conditioned system at all: by Propositions~\ref{prop:posterior_curvature} and~\ref{prop:ensf_curvature} and Figure~\ref{fig:cond}, it moves the ensemble through a family of progressively regularized intermediate distributions whose conditioning is controlled by the diffusion schedule and the damping function rather than fixed by $H$ and $R$. Conditioning enters the two routes at different levels, and this is the precise sense in which the transport realizes posterior-geometry regularization.

Third, the flow-based update is neither exact nor free of choices. Its guided score is an approximation away from $t=0$, and the Gaussian tests quantify the price: EnSF plateaus at KL values between $0.07$ and $0.33$ in the ill-conditioned cases, above the floor that the converged exact samplers reach. Its configuration including noise schedule, damping function, endpoint was fixed once and reused across all experiments. The reverse-step budget and ensemble size were set per experiment which shows the empirical robustness of the flow-based method.

Fourth, the Lorenz--96 benchmark shows that the same picture survives contact with a nonlinear, strongly heterogeneous system at $d=10^{4}$, where the exact posterior is unavailable and all methods share localized ensemble covariances and perturbed observations. There the flow-based update is simultaneously the most stable (Figure~\ref{fig:l96_rmse_all}), the best at preserving the attractor geometry (Figure~\ref{fig:l96_trajectory}), and roughly an order of magnitude cheaper at delivered accuracy (Figure~\ref{fig:l96_cost}), because the banded structure of the localized covariance can be exploited exactly inside the reverse-SDE solve while the CVT iteration is dominated by dense covariance products.

There are clear limits to the interpretation advanced here. The strongest mathematical statements are available in Gaussian or near-Gaussian settings, where the diffused covariance $\Sigma_t=\alpha_t^2 B_{n+1\mid n}+\beta_t^2 I$ makes the regularization mechanism explicit; outside that regime, the argument becomes local or heuristic, and the precise gap between the guided score and the exact posterior score remains an open issue. The cost comparison is implementation- and hardware-specific, and second-level-preconditioned Krylov solvers were not benchmarked. Overall, variational DA and flow-based transport should be viewed as distinct computational routes toward the same Bayesian target, and the experiments indicate that the transport route matters most when posterior geometry is dominated by anisotropy, weak or extremely precise observations, and tight iteration budgets.

\section{Conclusion}
\label{sec:conclusion}
This paper studied how variational optimization and flow-based transport realize the same Bayesian update through different computational geometries. Incremental variational DA repeatedly solves a fixed local optimization problem whose Hessian characterizes the local posterior geometry.
Flow-based DA follows a different route. Rather than repeatedly operating on the fixed endpoint curvature, it transports samples through a pseudo-time sequence of intermediate geometries connecting an isotropic Gaussian reference to the target posterior. A natural next step is to study how flow-based transport can be integrated into existing operational variational and hybrid assimilation systems as a conditioning layer, rather than as a replacement for the full variational framework. In this direction, one possible use is to apply the flow-based update inside selected outer loops or difficult assimilation windows, where the local incremental problem is strongly affected by covariance anisotropy, nonlinear observation operators, or a weak observational constraint. This would allow the method to complement existing control-variable transforms, localization, and hybrid covariance models while preserving the surrounding operational infrastructure. 
Developing such hybrid operational interfaces, together with scalable localization, adaptive diffusion schedules, efficient parallel implementations, and systematic comparisons against second-level-preconditioned Krylov solvers and control-space implementations of the transport, is an important direction for future research.

\section*{Acknowledgments}
This paper was supported by the U.S. Department of Energy (DOE), Office of Science, Office of Advanced Scientific Computing Research (ASCR), Applied Mathematics Program under contracts ERKJ443 and ERKJ388.
Feng Bao acknowledges support from the U.S. National Science Foundation (NSF) under grant DMS-2142672 and from the DOE Office of Science, ASCR Applied Mathematics Program under grant DE-SC0025412. Hristo G.~Chipilski acknowledges support from Florida State University’s CRC Seed Grant (047080). Peter Jan van Leeuwen acknowledges support from the National Oceanic and Atmospheric Administration (NOAA) project CADRE (NA24OARX459C0002-T1-01) and the Department of Defense (DOD) project RAM-HORNS (N00014-24-2017).

\bibliographystyle{siamplain}
\bibliography{mybib}

@article{Carrassi2018,
  author  = {Carrassi, Alberto and Bocquet, Marc and Bertino, Laurent and Evensen, Geir},
  title   = {Data Assimilation in the Geosciences: An Overview of Methods, Issues, and Perspectives},
  journal = {WIREs Climate Change},
  year    = {2018},
  volume  = {9},
  number  = {5},
  pages   = {e535}
}

@article{gaspari1999construction,
  title={Construction of correlation functions in two and three dimensions},
  author={Gaspari, Gregory and Cohn, Stephen E},
  journal={Quarterly Journal of the Royal Meteorological Society},
  volume={125},
  number={554},
  pages={723--757},
  year={1999},
  publisher={Wiley Online Library}
}

@article{ParrishDerber1992,
  author  = {Parrish, David F. and Derber, John C.},
  title   = {The National Meteorological Center's Spectral Statistical-Interpolation Analysis System},
  journal = {Monthly Weather Review},
  year    = {1992},
  volume  = {120},
  number  = {8},
  pages   = {1747--1763}
}

@article{Barker2004,
  author  = {Barker, Dale M. and Huang, Wei and Guo, Yong-Ran and Bourgeois, Amy J. and Xiao, Qingnong},
  title   = {A Three-Dimensional Variational Data Assimilation System for MM5: Implementation and Initial Results},
  journal = {Monthly Weather Review},
  year    = {2004},
  volume  = {132},
  number  = {4},
  pages   = {897--914}
}

@article{Rabier2000,
  author  = {Rabier, Florence and J{\"a}rvinen, H. and Klinker, E. and Mahfouf, Jean-Fran{\c c}ois and Simmons, Adrian},
  title   = {The {ECMWF} Operational Implementation of Four-Dimensional Variational Assimilation. {I}: Experimental Results with Simplified Physics},
  journal = {Quarterly Journal of the Royal Meteorological Society},
  year    = {2000},
  volume  = {126},
  number  = {564},
  pages   = {1143--1170}
}

@article{Huang2009,
  author  = {Huang, Xiang-Yu and Xiao, Qingnong and Barker, Dale M. and Zhang, Xin and Michalakes, John and Huang, Wei and Henderson, Thomas and Bray, John and Chen, Yongsheng and Ma, Zhaoxia and Dudhia, Jimy and Guo, Yong-Ran and Zhang, Xuhui and Won, Dong-Joo and Lin, Han-Ching and Kuo, Ying-Hwa},
  title   = {Four-Dimensional Variational Data Assimilation for {WRF}: Formulation and Preliminary Results},
  journal = {Monthly Weather Review},
  year    = {2009},
  volume  = {137},
  number  = {1},
  pages   = {299--314}
}

@article{Tremolet2007,
  author  = {Tr{\'e}molet, Yannick},
  title   = {Incremental 4D-Var Convergence Study},
  journal = {Tellus A: Dynamic Meteorology and Oceanography},
  year    = {2007},
  volume  = {59},
  number  = {5},
  pages   = {706--718}
}

@article{BaoEtAl2025,
  author  = {Bao, Feng and Chipilski, Hristo G. and Liang, Siming and Zhang, Guannan and Whitaker, Jeffrey S.},
  title   = {Nonlinear Ensemble Filtering with Diffusion Models: Application to the Surface Quasi-Geostrophic Dynamics},
  journal = {Monthly Weather Review},
  year    = {2025},
  volume  = {153},
  number  = {7},
  pages   = {1155--1169}
}

@article{ZhangBaoZhang2025IEnSF,
  author        = {Zhang, Zezhong and Bao, Feng and Zhang, Guannan},
  title         = {{IEnSF}: Iterative Ensemble Score Filter for Reducing Error in Posterior Score Estimation in Nonlinear Data Assimilation},
  journal       = {Journal of Computational Physics},
  year          = {2026},
  volume={568},
  pages={115357}
}

@article{Bao2024EnSF,
  author  = {Bao, Feng and Zhang, Zezhong and Zhang, Guannan},
  title   = {An Ensemble Score Filter for Tracking High-Dimensional Nonlinear Dynamical Systems},
  journal = {Computer Methods in Applied Mechanics and Engineering},
  volume  = {432},
  pages   = {117447},
  year    = {2024}
}

@article{CourtierEtAl1998,
  author  = {Courtier, Philippe and Andersson, Erik and Heckley, William and others},
  title   = {The ECMWF Implementation of Three-Dimensional Variational Assimilation (3D-Var). I: Formulation},
  journal = {Quarterly Journal of the Royal Meteorological Society},
  volume  = {124},
  number  = {550},
  pages   = {1783--1807},
  year    = {1998}
}

@article{DerberBouttier1999,
  author  = {Derber, John and Bouttier, Fran{\c{c}}ois},
  title   = {A Reformulation of the Background Error Covariance in the ECMWF Global Data Assimilation System},
  journal = {Tellus A},
  volume  = {51},
  number  = {2},
  pages   = {195--221},
  year    = {1999}
}

@article{FisherEtAl2009,
  author  = {Fisher, Mike and Nocedal, Jorge and Tr{\'e}molet, Yannick and Wright, Stephen J.},
  title   = {Data Assimilation in Weather Forecasting: A Case Study in PDE-Constrained Optimization},
  journal = {Optimization and Engineering},
  volume  = {10},
  number  = {3},
  pages   = {409--426},
  year    = {2009}
}

@article{GrattonTshimanga2009,
  author  = {Gratton, Serge and Tshimanga, Jean},
  title   = {An Observation-Space Formulation of Variational Assimilation Using a Restricted Preconditioned Conjugate-Gradient Algorithm},
  journal = {Quarterly Journal of the Royal Meteorological Society},
  volume  = {135},
  pages   = {1573--1585},
  year    = {2009}
}

@article{WangEtAl2013,
  author  = {Wang, Xuguang and Parrish, David F. and Kleist, Daryl T. and Whitaker, Jeffrey S.},
  title   = {GSI 3DVar-Based Ensemble-Variational Hybrid Data Assimilation for NCEP Global Forecast System: Single-Resolution Experiments},
  journal = {Monthly Weather Review},
  volume  = {141},
  number  = {11},
  pages   = {4098--4117},
  year    = {2013}
}

@article{LorenzEmanuel1998,
  author  = {Lorenz, Edward N. and Emanuel, Kerry A.},
  title   = {Optimal Sites for Supplementary Weather Observations: Simulation with a Small Model},
  journal = {Journal of the Atmospheric Sciences},
  volume  = {55},
  number  = {3},
  pages   = {399--414},
  year    = {1998}
}

@article{hamill2000hybrid,
  author  = {Hamill, Thomas M. and Snyder, Chris},
  title   = {A Hybrid Ensemble {Kalman} Filter--3D Variational Analysis Scheme},
  journal = {Monthly Weather Review},
  year    = {2000},
  volume  = {128},
  number  = {8},
  pages   = {2905--2919}
}

@article{lorenc2003potential,
  author  = {Lorenc, Andrew C.},
  title   = {The Potential of the Ensemble {Kalman} Filter for {NWP}---A Comparison with 4D-Var},
  journal = {Quarterly Journal of the Royal Meteorological Society},
  year    = {2003},
  volume  = {129},
  number  = {595},
  pages   = {3183--3203}
}

@article{buehner2005ensemble,
  author  = {Buehner, Mark},
  title   = {Ensemble-Derived Stationary and Flow-Dependent Background-Error Covariances:
             Evaluation in a Quasi-Operational {NWP} Setting},
  journal = {Quarterly Journal of the Royal Meteorological Society},
  year    = {2005},
  volume  = {131},
  number  = {607},
  pages   = {1013--1043}
}

@article{bannister2017,
  author  = {Bannister, Ross N.},
  title   = {A Review of Operational Methods of Variational and Ensemble-Variational Data Assimilation},
  journal = {Quarterly Journal of the Royal Meteorological Society},
  year    = {2017},
  volume  = {143},
  number  = {703},
  pages   = {607--633}
}

@article{xiong2025sensitivity,
    title={Robustness of the Ensemble Score Filter to the Type of Assimilated Observation Networks},
  author={Xiong, Zixiang and Liang, Siming and Bao, Feng and Zhang, Guannan and Chipilski, Hristo G},
  journal={Atmospheric Science Letters},
  volume={27},
  number={1},
  pages={e70004},
  year={2026},
  publisher={Wiley Online Library}
}

@article{liang_ensf_inpainting_2025,
  title={Ensemble score filter with image inpainting for data assimilation in tracking surface quasi-geostrophic dynamics with partial observations},
  author={Liang, Siming and Tran, Hoang and Bao, Feng and Chipilski, Hristo G and van Leeuwen, Peter Jan and Zhang, Guannan},
  journal={arXiv preprint arXiv:2501.12419},
  year={2025}
}

@article{VanLeeuwen1996,
author = "{van Leeuwen}, Peter Jan and Geir  Evensen",
title = "Data Assimilation and Inverse Methods in Terms of a Probabilistic Formulation",
journal = "Monthly Weather Review",
year = "1996",
publisher = "American Meteorological Society",
address = "Boston MA, USA",
volume = "124",
number = "12",
pages=      "2898-2913"
}

@article{GrattonLawlessNichols2007,
  author  = {Gratton, Serge and Lawless, Amos S. and Nichols, Nancy K.},
  title   = {Approximate {G}auss--{N}ewton Methods for Nonlinear Least Squares Problems},
  journal = {SIAM Journal on Optimization},
  volume  = {18},
  number  = {1},
  pages   = {106--132},
  year    = {2007}
}

@article{TshimangaEtAl2008,
  author  = {Tshimanga, Jean and Gratton, Serge and Weaver, Anthony T. and Sartenaer, Annick},
  title   = {Limited-Memory Preconditioners, with Application to Incremental Four-Dimensional Variational Data Assimilation},
  journal = {Quarterly Journal of the Royal Meteorological Society},
  volume  = {134},
  number  = {632},
  pages   = {751--769},
  year    = {2008}
}

@article{HabenLawlessNichols2011,
  author  = {Haben, Stephen A. and Lawless, Amos S. and Nichols, Nancy K.},
  title   = {Conditioning of Incremental Variational Data Assimilation, with Application to the {M}et {O}ffice System},
  journal = {Tellus A},
  volume  = {63},
  number  = {4},
  pages   = {782--792},
  year    = {2011}
}

@article{Burgers1998,
  author  = {Burgers, Gerrit and van Leeuwen, Peter Jan and Evensen, Geir},
  title   = {Analysis Scheme in the Ensemble {K}alman Filter},
  journal = {Monthly Weather Review},
  volume  = {126},
  number  = {6},
  pages   = {1719--1724},
  year    = {1998}
}

@article{Bardsley2014RTO,
  author  = {Bardsley, Johnathan M. and Solonen, Antti and Haario, Heikki and Laine, Marko},
  title   = {Randomize-Then-Optimize: A Method for Sampling from Posterior Distributions in Nonlinear Inverse Problems},
  journal = {SIAM Journal on Scientific Computing},
  volume  = {36},
  number  = {4},
  pages   = {A1895--A1910},
  year    = {2014}
}

@article{GilbertLemarechal1989,
  author  = {Gilbert, Jean Charles and Lemar{\'e}chal, Claude},
  title   = {Some Numerical Experiments with Variable-Storage Quasi-{N}ewton Algorithms},
  journal = {Mathematical Programming},
  volume  = {45},
  pages   = {407--435},
  year    = {1989}
}

@article{GurolEtAl2014,
  author  = {G{\"u}rol, Selime and Weaver, Anthony T. and Moore, Andrew M. and Piacentini, Andrea and Arango, Hernan G. and Gratton, Serge},
  title   = {{B}-Preconditioned Minimization Algorithms for Variational Data Assimilation with the Dual Formulation},
  journal = {Quarterly Journal of the Royal Meteorological Society},
  volume  = {140},
  number  = {679},
  pages   = {539--556},
  year    = {2014}
}

@article{Spantini2022Coupling,
  author  = {Spantini, Alessio and Baptista, Ricardo and Marzouk, Youssef},
  title   = {Coupling Techniques for Nonlinear Ensemble Filtering},
  journal = {SIAM Review},
  volume  = {64},
  number  = {4},
  pages   = {921--953},
  year    = {2022}
}

@article{MorzfeldHodyss2023,
  author  = {Morzfeld, Matthias and Hodyss, Daniel},
  title   = {A Theory for Why Even Simple Covariance Localization Is So Useful in Ensemble Data Assimilation},
  journal = {Monthly Weather Review},
  volume  = {151},
  number  = {3},
  pages   = {717--736},
  year    = {2023}
}

@article{Tong2018LocalEnKF,
  author  = {Tong, Xin T.},
  title   = {Performance Analysis of Local Ensemble {K}alman Filter},
  journal = {Journal of Nonlinear Science},
  volume  = {28},
  number  = {4},
  pages   = {1397--1442},
  year    = {2018}
}

@article{LiGelbLee2024,
  author  = {Li, Tongtong and Gelb, Anne and Lee, Yoonsang},
  title   = {A Structurally Informed Data Assimilation Approach for Nonlinear Partial Differential Equations},
  journal = {Journal of Computational Physics},
  volume  = {519},
  pages   = {113396},
  year    = {2024}
}

@book{NocedalWright2006,
  author    = {Nocedal, Jorge and Wright, Stephen J.},
  title     = {Numerical Optimization},
  edition   = {2nd},
  series    = {Springer Series in Operations Research and Financial Engineering},
  publisher = {Springer},
  address   = {New York},
  year      = {2006}
}

@inproceedings{ColeLu2024,
  author    = {Cole, Frank and Lu, Yulong},
  title     = {Score-Based Generative Models Break the Curse of Dimensionality in Learning a Family of Sub-{G}aussian Distributions},
  booktitle = {International Conference on Learning Representations},
  year      = {2024}
}

@article{Poterjoy2016LPF,
  author  = {Poterjoy, Jonathan},
  title   = {A Localized Particle Filter for High-Dimensional Nonlinear Systems},
  journal = {Monthly Weather Review},
  volume  = {144},
  number  = {1},
  pages   = {59--76},
  year    = {2016}
}

@article{Transue2026FlowMatchingDA,
  author  = {Transue, Taos and Chen, Bohan and Takao, So and Wang, Bao},
  title   = {Flow Matching for Efficient and Scalable Data Assimilation},
  journal = {SIAM/ASA Journal on Uncertainty Quantification},
  year    = {2026},
  note    = {To appear; arXiv:2508.13313}
}

@inproceedings{si2025latent,
  title={Latent-EnSF: A latent ensemble score filter for high-dimensional data assimilation with sparse observation data},
  author={Si, Phillip and Chen, Peng},
  booktitle={International Conference on Learning Representations},
  volume={2025},
  pages={94879--94895},
  year={2025}
}

@inproceedings{xiao2026ld,
  title={LD-EnSF: synergizing latent dynamics with ensemble score filters for fast data assimilation with sparse observations},
  author={Xiao, Pengpeng and Si, Phillip and Chen, Peng},
  booktitle={International Conference on Learning Representations},
  volume={2026},
  pages={152461--152491},
  year={2026}
}

@misc{BinderDasguptaOberai2026,
  author        = {Binder, Brianna and Dasgupta, Agnimitra and Oberai, Assad},
  title         = {Closed-Form Conditional Diffusion Models for Data Assimilation},
  year          = {2026},
  howpublished  = {arXiv:2603.21291}
}

\end{document}